\documentclass[reqno,a4paper,dvipsnames]{amsart}
 \usepackage{amsgen, amstext,amsbsy,amsopn, amsthm, amsfonts,amssymb,amscd,amsmath,euscript,enumerate,url,verbatim,calc,tikz}

\usepackage{MnSymbol}
\usepackage{pgfkeys}
\usepackage{mathtools}
\usepackage[left=1.2in,right=1.2in,bottom=1.5in]{geometry}
\usepackage{enumerate}
\usepackage{relsize}
\usepackage{pst-node}

\usepackage{xcolor}
\usepackage{color}
\colorlet{linkequation}{cyan}
\usepackage[colorlinks=true,linkcolor=blue]{hyperref}
\usepackage{cleveref}

\usepackage{tikz-cd}

\usepackage{eqnarray,amsmath}

\def\multiset#1#2{\ensuremath{\left(\kern-.3em\left(\genfrac{}{}{0pt}{}{#1}{#2}\right)\kern-.3em\right)}}

\usetikzlibrary{arrows}

 \usepackage{latexsym}
 \usepackage{graphics}
\usepackage{lastpage}
\usepackage{fancyhdr}
\usepackage{multirow}
\allowdisplaybreaks
\usepackage{graphicx}
\graphicspath{ {F:/IMAGES/} }

\makeatletter
\def\oversortoftilde#1{\mathop{\vbox{\m@th\ialign{##\crcr\noalign{\kern3\p@}%
      \sortoftildefill\crcr\noalign{\kern3\p@\nointerlineskip}%
      $\hfil\displaystyle{#1}\hfil$\crcr}}}\limits}

\def\sortoftildefill{$\m@th \setbox\z@\hbox{$\braceld$}%
  \braceld\leaders\vrule \@height\ht\z@ \@depth\z@\hfill\braceru$}

\makeatother

  \newcommand{\type}{\operatorname{type}}

  \newcommand{\reg}{\operatorname{reg}}

 \newcommand{\ex}{\operatorname{end}}

 \newcommand{\grade}{\operatorname{grade}}
 \newcommand{\depth}{\operatorname{depth}}

 \newcommand{\Ext}{\operatorname{Ext}}

\newcommand{\proset}{\,\mathrel{\lower 4pt\hbox{$\scriptscriptstyle/$}
\mkern -14mu\subseteq }\,} 

 \newtheorem{theorem}{Theorem}[section]
 \newtheorem{corollary}[theorem]{Corollary}
 \newtheorem{lemma}[theorem]{Lemma}
 \newtheorem{proposition}[theorem]{Proposition}

 \newtheorem{question}[theorem]{Question}

\usepackage{amsmath}
\newtheorem{notation}[theorem]{Notation}

 \theoremstyle{definition}
 
 \newtheorem{remark}[theorem]{Remark}
 \newtheorem{definition}[theorem]{Definition}

 \newtheorem{example}[theorem]{Example}

\allowdisplaybreaks

\title[Bounds on the  Ratliff-Rush Index and the Castelnuovo-Mumford Regularity] {Bounds on the  Ratliff-Rush Index and the Castelnuovo-Mumford Regularity}

\author{Mousumi Mandal and Shruti Priya }
\date{}

\thanks{AMS Classification 2010: 13H10, 13D40, 13A30.}
\address{Department of Mathematics, Indian Institute of Technology Kharagpur, 721302, India} \email{mousumi@maths.iitkgp.ac.in}
\address{Department of Mathematics, Indian Institute of Technology Kharagpur, 721302, India} \email{shruti96312@kgpian.iitkgp.ac.in}

\begin{document}  
\allowdisplaybreaks

\begin{abstract}

Let $(R, \mathfrak m)$ be a Cohen-Macaulay local ring of dimension $d \geq 2,$ and $I$ an $\mathfrak m$-primary ideal of $R.$  Denote $r_{J}(I)$ as the reduction number of $I$ with respect to a minimal reduction $J$ of $I,$ and $\rho(I)$ as the Ratliff-Rush index of $I$. We establish upper bounds on $\rho(I)$ in terms of Hilbert coefficients $e_{i}(I)$ for $0 \leq i \leq d+1,$ and $r_{J}(I).$ Suppose $\widetilde{I^{r_{J}(I)}} \neq I^{r_{J}(I)}.$ We prove that $\rho(I) \leq r_{J}(I)-1+(-1)^{d+1}(e_{d+1}(I)-\widetilde{e}_{d+1}(I)).$  When $d=2,$ we prove that $\rho(I) \leq r_{J}(I) -1 +(e_{2}(I)-1)e_{2}(I)-e_{3}(I).$ This established bound on $\rho(I)$ consequently leads to a bound on the Castelnuovo-Mumford regularity of the associated graded ring of $I.$ We also determine bound on $\rho(I)$ in  two-dimensional Buchsbaum rings with positive depth.  
\end{abstract} \maketitle
\vspace{-1em}
\begin{center}
    \scriptsize{\emph{Dedicated to Professor Jugal Verma on the occasion of his 65th birthday}}
\end{center}

\section{Introduction}
The concept of Castelnuovo-Mumford regularity or simply regularity  was first presented in the early 1980s by Eisenbud and Goto \cite{EG}, and Ooishi \cite{Ooishi} as an algebraic equivalent to the concept of regularity for coherent sheaves on projective spaces, which was formerly explored by Mumford \cite{dm}. It is a kind of universal bound for important invariants of graded algebras  such as the maximum degree of the syzygies and the maximum non-vanishing degree of the local cohomology modules.

The concept of regularity  can also be defined for graded structures over a commutative ring
 like the Rees algebra $\mathcal{R}(I)=\displaystyle \bigoplus_{n \geq 0}I^{n}t^{n} \subseteq R[t]$, where $t$ is an indeterminate over $R$  and the associated graded ring $G(I)=\displaystyle \bigoplus_{n \geq 0} I^{n}/I^{n+1}$ of $R$ with respect to $I$.  
The relation between  $\reg \mathcal{R}(I)$ and $\reg G(I)$ was well studied by numerous authors in the past (see  \cite{ooishi2, trung}). 

Efforts are frequently  made to determine  its upper bounds  using simpler invariants which are computationally easier to determine. The most fundamental invariants that reflect the complexity of a graded algebra are its dimension and multiplicity.  Nevertheless, it is not easy to establish a simpler upper bound only based on the multiplicity and dimension. Previous results of Srinivas and Trivedi \cite{st1,st2}, and Trivedi \cite{trivedi, trivedi2} gave bounds on  regularity in terms of dimension and multiplicity for Cohen-Macaulay rings and generalized Cohen-Macaulay modules. A celebrated result by Rossi, Trung and Valla  \cite[Theorem 3.3]{rtv}, gave a bound on the regularity of the associated graded ring  with respect to the the maximal ideal in a local ring involving extended degree, dimension and multiplicity. 
  For  Cohen-Macaulay local rings  \cite[Corollary 3.4]{rtv},  authors  proved that 
 \begin{equation*}
\reg G(\mathfrak m) \leq
    \begin{cases}
        e(\mathfrak m)-1, & \text{if } d=1\\
          e(\mathfrak m)^{2((d-1)!)-1}[e(\mathfrak m)-1]^{(d-1)!}, &  \text{if } d\geq 2 . 
    \end{cases}
\end{equation*}
 Linh \cite[Theorem 4.4]{chl} further extended these bounds  to the associated graded module of an arbitrary finitely generated graded $R$-module with respect to an arbitrary $\mathfrak m$-primary ideal. There are various bounds on $\reg G(I)$ in terms of other invariants of $I$ (see \cite{dgv, dhoa, str, vas}).

   Recall that in a Noetherian ring $R,$  the ascending chain of ideals $I\subseteq(I^{2}:I)\subseteq(I^{3}:I^{2})\subseteq\ldots\subseteq(I^{n+1}:I^{n})\subseteq\ldots$
 eventually stabilizes to an ideal $\widetilde{I}=\displaystyle \bigcup_{n \geq 1}(I^{n+1}:I^{n})$. This ideal $\widetilde{I}$ is known as the \textit{Ratliff-Rush closure} of $I$ or the \textit{Ratliff-Rush ideal} associated with $I.$ An ideal $I$ is termed  \textit{Ratliff-Rush closed} if $\widetilde{I}=I.$  The \textit{Ratliff-Rush filtration} with respect to $I$, is the filtration $\mathcal{F}=\{\widetilde{I^{n}}\}_{n \geq 0},$ and if $\grade(I)>0$, then $\widetilde{I^{n}}=I^{n}$ for sufficiently large $n.$ For an $\mathfrak m$-primary ideal $I,$ we define
 \begin{equation*}
     \rho(I)=\min\{n \in \mathbb{N}: \widetilde{I^{k}}=I^{k} \text{ for all } k \geq n\}.
 \end{equation*} 

 This numerical invariant is known as  the \textit{Stability index} of the Ratliff-Rush filtration or the \textit{Ratliff-Rush Index.}
  The principal goal of this paper is to introduce   upper bounds on the Castelnuovo-Mumford regularity using higher Hilbert coefficients (see Definition \ref{def1}).  The starting point of our approach is the description of the regularity by Rossi, Trung and Trung  \cite[Theorem 2.4]{rtt}. The authors proved that for a two-dimensional 
  Buchsbaum ring $R$ with $\depth R >0$, $I$ an $\mathfrak m$-primary ideal,  and $J$ a minimal reduction of $I$, 
      $\reg G(I)=\reg \mathcal{R}(I)=\max \{r_{J}(I),\rho(I)\}.$

  To establish an upper bound on the Castelnuovo-Mumford regularity in dimension two, it is crucial to determine effective bounds on both $r_{J}(I)$ and $\rho(I)$. A remarkable result of Rossi \cite[Corollary 1.5]{rossi2}, gave bound on $r_{J}(I)$ in Cohen-Macaulay local rings of dimension at most two. 
Consequently, if $\reg G(I) =\rho(I)$, then it is essential to give a  bound on $\rho(I).$ The computational task of determining $\rho(I)$ poses significant challenges even in dimension two, just because $\widetilde{I^{n}}=I^{n}$ does not imply $\widetilde{I^{n+1}}=I^{n+1}.$   Thus, our aim is to establish bounds on $\rho(I)$ for Cohen-Macaulay local rings of dimension $d$.

Huneke \cite[Proposition 4.3]{hjls} established a significant relationship between $\rho(I)$ and $r_{J}(I)$ in dimension  two. He proved that  for a   Cohen-Macaulay local ring $(R,\mathfrak m)$ of dimension $d = 2,$ $I$  an $\mathfrak m$-primary ideal and $J \subseteq I$  a minimal reduction, if $\widetilde{I^{m}}=I^{m}$ for some $m \geq r_{J}(I)$, then $\widetilde{I^{n}}=I^{n}$ for all  $n \geq m.$ Building on this foundation, we first extend this result to 
$d$-dimensional Cohen-Macaulay local rings (see Proposition \ref{maa}). As a direct consequence of this generalization, we establish the following theorem:
  \begin{theorem} \label{nehal}
      Let $R$ be a Cohen-Macaulay local ring of dimension  $d \geq 2$, $I$ an $\mathfrak m$-primary ideal and $J \subseteq I$ be a minimal reduction of $I$. 
      If $\widetilde{I^{r_{J}(I)}} \neq I^{r_{J}(I)},$ then 
      \begin{equation*}
          \rho(I) \leq r_{J}(I)-1+(-1)^{d+1}(e_{d+1}(I)-\widetilde{e}_{d+1}(I)).
      \end{equation*}
   \end{theorem}
We observe that under the  assumption that  the Ratliff-Rush filtration with respect to $I$ behaves well mod a superficial sequence $x_{1}, x_{2}, \ldots, x_{d-2}  \in I$ (see Definition \ref{rrr}) in Theorem \ref{nehal},  if the dimension of the ring is odd then $\rho(I) \leq r_{J}(I)-1+e_{d+1}(I),$ and if the dimension of the ring is even then  $\rho(I) \leq r_{J}(I)-1-e_{d+1}(I)+\left(\frac{r_{J}(I)-1}{d}\right)(e_{d}(I)-e_{d-1}(I)+\cdots+e(I)-\lambda(R/I)+\lambda(\widetilde{I}/I)).$ 
For two-dimensional Cohen-Macaulay local rings, we prove the following bound on $\rho(I):$

\begin{proposition} \label{prop2}
    Let $(R,\mathfrak m)$ be a  two dimensional Cohen-Macaulay local ring  and $I$ an $\mathfrak m$-primary ideal. For a minimal reduction $J$ of $I$, if $r_{J}(I) < \rho(I),$ then  
    \begin{equation} \label{xxx}
        \rho(I) \leq r_{J}(I)-1+(e_{2}(I)-1)e_{2}(I)-e_{3}(I).
    \end{equation}
\end{proposition}
Though the bound in (\ref{xxx}) is quadratic in $e_{2}(I),$ we illustrate various examples where it is tighter than the earlier known bounds.

One of the generalizations of a Cohen-Macaulay ring is a Buchsbaum ring. In such rings, the condition of a system of parameters being a regular sequence is relaxed to being a weak sequence. Recall that a Noetherian local ring $(R, \mathfrak m)$ is called a Buchsbaum ring if every system of parameter $x_{1},\ldots,x_{r}$ of $R$ is a weak sequence, i.e., \begin{center}
     $(x_{1},\ldots,x_{i-1}):x_{i}=(x_{1},\ldots,x_{i-1}):\mathfrak m \text{ for } i=1,\ldots, r.$
 \end{center}
Therefore, it is worth asking if we can find a similar bound on $\rho(I)$ for Buchsbaum rings of positive depth. Thus, we prove the following result:
\begin{theorem} \label{rudra1}
    Let $(R, \mathfrak m)$ be a two dimensional Buchsbaum ring with  $\depth R >0$  and $\mathbb S$ be the $S_{2}$-fication of $R$. Let $I$ be an $\mathfrak m$-primary ideal with minimal reduction $J$ and assume that $I=I\mathbb S$. If $r_{J}(I) < \rho(I),$then $\rho(I) \leq r_{J}(I)-1+(e_{2}(I)-e_{1}(J)-1)(e_{2}(I)-e_{1}(J))-e_{3}(I)$.
\end{theorem}

As an application, the bounds we establish on $\rho(I)$  enables us to derive corresponding bounds on $\reg G(I)$ in two-dimensional Cohen Macaulay local rings, when $r_{J}(I) < \rho(I)$ (see Proposition \ref{lama} and Proposition \ref{sp}).

This paper is organized into five sections. Definitions, introductory ideas and notations are all covered in Section 2.    In Section 3, we  establish bounds on $\rho(I)$ and  prove Theorem 1.1 and Proposition 1.2.   Further in this section, we gather various bounds on $\rho(I)$ only in terms of $r_{J}(I).$  In Section 4, we prove Theorem 1.3 and establish bounds on $\rho(I)$ in two-dimensional Buchsbaum rings with positive depth.
Section 5 is an application of the results of Sections 3 and 4, where we establish bounds  on the regularity. We provide examples to support our claims.

\section{Preliminaries}
In this section, we provide an overview of the definitions and relevant facts that will be used throughout the paper, as well as establish some conventions and notations. Throughout, $(R, \mathfrak m)$ will be a Cohen-Macaulay local ring (\textit{unless stated otherwise}) with an infinite residue field and $I$ an $\mathfrak m$-primary ideal of $R$. The Krull dimension of the ring $R$  will be denoted by $d$ and we will  always assume that $d \geq 2$. For undefined terms, we redirect readers to \cite{bh}.

\begin{definition} \label{def1}
    A sequence of ideals $\mathcal{I}=\{I_{n}\}_{n\in \mathbb Z}$ is called an $I$-\textit{admissible filtration}, if for all $n, m \in \mathbb Z,$ $(i)$ $I_{n+1} \subseteq I_{n},$ $(ii)$ $I_{m}I_{n} \subseteq I_{m+n}$ and $(iii)$ $I^{n} \subseteq I_{n} \subseteq I^{n-k}$ for some $k \in \mathbb N.$ The \textit{first Hilbert function} or the \textit{Hilbert-Samuel function} of $\mathcal{I}$ , $H^{1}_{\mathcal{I}}(n):= \lambda(R/I_{n+1})$ (where $\lambda()$ denotes the length) agrees with  a polynomial $P^{1}_{\mathcal{I}}(n)$  for sufficiently large  $n$. This polynomial, termed the \textit{first Hilbert polynomial} or the \textit{Hilbert-Samuel polynomial} of $\mathcal{I}$, assumes only numerical values for all integers $n$. For simplicity we write  $H^{1}_{\mathcal{I}}(n)$ as $H_{\mathcal{I}}(n)$ and $P^{1}_{\mathcal{I}}(n)$ as $P_{\mathcal{I}}(n)$. By the theory on numerical polynomials we can write 
\begin{equation*}
    P_{\mathcal{I}}(n)=e_{0}(\mathcal{I})\dbinom{n+d}{d}-e_{1}(\mathcal{I})\dbinom{n+d-1}{d-1}+\cdots+(-1)^{d}e_{d}(\mathcal{I}).
\end{equation*}
 The unique integers $e_{0}(\mathcal{I}),e_{1}(\mathcal{I}), \ldots, e_{d}(\mathcal{I})$ are referred to as the \textit{Hilbert  coefficients.}
 When $\mathcal{I}$ is the $I$-adic filtration $\{I^{n}\}_{n \geq 0}$, we write $H_{I}(n)$ and $P_{I}(n)$ in the place of $H_{\mathcal{I}}(n)$ and $P_{\mathcal{I}}(n)$ respectively.
 The leading coefficient $e_{0}(I)$ of the polynomial $P_{I}(n)$ is called the \textit{multiplicity} of the ideal $I$, and is denoted by $e(I).$  We write $\widetilde{e}_{i}(I)$ for the coefficients $e_{i}(\mathcal{I})$, when $\mathcal{I}= \mathcal{F}=\{\widetilde{I^{n}}\}_{n \geq 0}.$
\end{definition}

\begin{definition}
    The \textit{Hilbert series} of $\mathcal{I}$,  is the formal power series $\displaystyle \sum_{n \geq 0} \lambda(I_{n}/I_{n+1}) z^{n}$. 
    By Hilbert-Serre theorem, we write 
    \begin{equation*}
    \displaystyle \sum_{n \geq 0} \lambda(I_{n}/I_{n+1})z^{n}=\frac{h_{\mathcal{I}}(z)}{(1-z)^{d}},
\end{equation*}
where $h_{\mathcal{I}}(z) \in \mathbb{Z}[z]$ is the unique polynomial with $h_{\mathcal{I}}(1) \neq 0$, known as the \textit{h-polynomial} of $\mathcal{I}.$   The power series, $\displaystyle \sum_{n \geq 0} H_{\mathcal{I}}(n)z^{n}$ is called the \textit{Hilbert-Samuel series} of $\mathcal{I}.$ Note that $\displaystyle \sum_{n \geq 0} H_{\mathcal{I}}(n)z^{n}=\frac{h_{\mathcal{I}}(z)}{(1-z)^{d+1}},$
and an easy computation shows that for all $i \geq 0$
\begin{equation*}
     e_{i}(\mathcal{I})=\frac{h^{(i)}_{\mathcal{I}}(1)}{i!},
\end{equation*}
 where $h^{(i)}_{\mathcal{I}}(1)$ is the $i^{th}$-formal derivative of $h_{\mathcal{I}}(z)$ at $z=1.$ 
The conventional literature on local rings concentrates only on the integers $e_{i}(\mathcal{I})$ with $0 \leq i \leq d.$ However, more apparent results can be achieved by including the integers $e_{i}(\mathcal{I})$ with $ i > d.$ 
\end{definition}

\begin{definition}
     The \textit{second Hilbert function} of $\mathcal{I}$, defined as $H^{2}_{\mathcal{I}}(n):=\displaystyle\sum_{j=0}^{n}\lambda(R/I_{j+1})$ coincides with a polynomial $P^{2}_{\mathcal{I}}(n)$ known as the \textit{second Hilbert polynomial} for all sufficiently large values of $n$. This is a polynomial of degree $d+1, $ and can be formulated as
 \begin{equation*}
       P^{2}_{\mathcal{I}}(n)=\displaystyle \sum_{j=0}^{d+1}(-1)^{j}e_{j}(\mathcal{I})\binom{n+d-j+1}{d-j+1}.
 \end{equation*}
       
\end{definition}

 \begin{definition} \label{kkk}
     An ideal $J \subseteq I$ is a reduction of $I$ if $JI^{n}=I^{n+1}$ for some $n \in \mathbb{N.}$ 
 If $J$ is a reduction of $I$, the \textit{reduction number} of $I$ with respect to $J,$ denoted by  $r_{J}(I),$ is the smallest $n$ such that $JI^{n}=I^{n+1}.$
A reduction is  \textit{minimal} if it is minimal with respect to inclusion among all reductions. The \textit{reduction number} of $I$ is defined  as 
\begin{equation*}
   r(I):=\min\{r_{J}(I): J \text{ is a minimal reduction of } I\}. 
\end{equation*}
The concept of reduction and minimal reduction was first introduced by Northcott and Rees. 
Reduction number is useful in studying the Cohen-Macaulayness of the associated graded ring $G(I)$.

 \end{definition}

 \begin{definition} \label{phn}
        An element $x\in I \backslash I^2$ is called a \textit{superficial element} for $I,$ if there exist a positive integer $c$ such that $(I^{n+1}:x) \cap I^{c}=I^{n}$ for all $n \geq c.$  A sequence $x_{1},x_{2},\ldots,x_{j}$ is  a \textit{superficial sequence} for $I,$ if $x_{1}$ is a superficial element for $I$, and $x_{i}$ is a superficial element for $I/(x_{1},x_{2},\ldots,x_{i-1})$ for all $i=2,\ldots,j.$ 
    If $x $ is a superficial element for $I$, then $(I^{n+1}:x)=I^{n}$ for sufficiently large $n.$ Let $x$ be any superficial element for $I,$ then the \textit{stability index} for $x$ is defined as
    \begin{equation*}
         \rho_{x}(I):=\min\{k \in \mathbb{N}: (I^{n+1}:x)=I^{n}
  \text{ for all } n \geq k\}.
     \end{equation*}
     In \cite[Corollary 2.7]{tjp}, Puthenpurakal  proved that $\rho(I)=\rho_{x}(I)$, and thus $\rho_{x}(I)$ is independent of the choice of the superficial element. 
\end{definition}

\begin{notation} \normalfont
         \textbf{I.}  For a graded module $M=
\displaystyle \bigoplus_{n \geq 0}M_{n}$,  over a graded ring $A=
\displaystyle \bigoplus_{n \geq 0}A_{n}$, $M_{n}$ denotes its $n$-th graded piece and $M(-j)$ denotes the graded module $M$ shifted  by degree $j$.

\textbf{II.}  Set $\mathcal{R}(I)_{+}=\displaystyle \bigoplus_{n \geq 1}I^{n}t^{n}$ and $\mathfrak M=\mathfrak M_{\mathcal{R}(I)}= \mathfrak m \bigoplus \mathcal{R}(I)_{+}.$

\textbf{III.}   Set $L^{I}(R)=\displaystyle \bigoplus_{n\geq 0} R/I^{n+1}.$ In \cite{tjp}, Puthenpurakal proved that $L^{I}(R)$ is an $\mathcal{R}(I)$-module which is not finitely generated.   See \cite{tjp,tjp2} for reference and further applications of $L^{I}(R)$ module.

  \textbf{IV.}    For a positive integer $s \leq d-1$, let $x_{1}, \ldots, x_{s} \in I$ be a superficial sequence for $I$. Set $R_{i}=R/(x_{1},\ldots, x_{i})$ for all $i=1,\ldots,s.$ We have the natural map of quotients $\pi_{i,j}:R_{i}\longrightarrow R_{j}$, for every $0\leq i < j\leq s.$ Note that 
$\pi_{i,j}\left(\widetilde{I^{n}R_{i}} \right) \subseteq \widetilde{I^{n}R_{j}} 
$  for all $n \geq 1.$
This map induces the  map
$\pi_{i,j}^{n}:\dfrac{\left(\widetilde{I^{n}R_{i}}\right)}{\left(I^{n}R_{i}\right)} \longrightarrow \dfrac{\left(\widetilde{I^{n}R_{j}}\right)}{\left(I^{n}R_{j}\right)}$ for every $n \geq 1.$ 

For a fixed $n \geq 1,$ we observe that $\pi_{i,j}\left(\widetilde{I^{n}R_{i}} \right) =\widetilde{I^{n}R_{j}} $   if and only in $\pi_{i,j}^{n}$ is surjective. For $0 \leq i<j<k \leq s$, we have the following commutative diagram:
\begin{center}
    $\begin{tikzcd}
	{\widetilde{I^{n}R_{i}}/I^{n}R_{i}} \\
	\\
	{\widetilde{I^{n}R_{j}}/I^{n}R_{j}} && {\widetilde{I^{n}R_{k}}/I^{n}R_{k}}
	\arrow["{\pi_{i,j}^{n}}"', from=1-1, to=3-1]
	\arrow["{\pi_{i,k}^{n}}", from=1-1, to=3-3]
	\arrow["{\pi_{j,k}^{n}}"', from=3-1, to=3-3]
\end{tikzcd}$
\end{center}
\end{notation}

\begin{definition} \cite[Definition 4.4]{tjp2} \label{rrr}
  If $\pi_{i,j}^{n}$ and $\pi_{j,k}^{n}$ are surjective, then $\pi_{i,k}^{n}$ is also surjective. We say that the Ratliff-Rush filtration with respect to $I$ \textit{behaves well mod the superficial sequence} $x_{1},\ldots, x_{s} \in I ,$ if $\pi_{0,s}(\widetilde{I^{n}})=\widetilde{I^{n}R_{s}}$ for all $n \geq 1.$
\end{definition}

\begin{definition} 
 Let $x_{1},\ldots, x_{s}$ be a superficial sequence for $I$, set $R_{i}=R/(x_{1},\ldots, x_{i})$ for all $i=1,\ldots,s.$
Then for each $i=1,\ldots,s$, we define
\begin{equation*}
    \rho\left(I/(x_{1}, \ldots, x_{i})\right):=\min\{k \in \mathbb{N} :I^{n}R_{i}=\widetilde{I^{n}R_{i}} \text{ for all }  n \geq k \}.
\end{equation*}
\end{definition}

 \begin{remark} \label{21}

 \begin{enumerate}[(i)]
     \item  As evident by \cite[Remark 4.6]{tjp2}, if the Ratliff-Rush filtration with respect to $I$ behaves well mod a superficial sequence $x_{1}, x_{2}, \ldots x_{s}  \in I$ for some $s \leq d-1$, then $\depth \widetilde{G}(I)\geq s+1$, where $\widetilde{G}(I)=\displaystyle\bigoplus_{n \geq 0} \widetilde{I^{n}}/\widetilde{I^{n+1}}$ is the associated graded ring of the Ratliff-Rush filtration with respect to $I$.

     \item By \cite[Lemma 2.14]{tjm}, for  Cohen-Macaulay local rings of dimension $d \geq 2$ and $I$ an $\mathfrak m$-primary ideal with minimal reduction $J \subseteq I$,  if the Ratliff-Rush filtration with respect to $I$ behaves well mod a superficial element $x \in J$ for $I,$ then $\widetilde{r}_{J}(I)=\widetilde{r}_{J/(x)}(I/(x)),$  where $\widetilde{r}_{J}(I)=\min\{n \in \mathbb{N} : J\widetilde{I^{k}}=\widetilde{I^{k+1}} \text{ for all } k \geq n \}.$  

     \item For a Cohen-Macaulay local ring of dimension $d \geq 3$,  $I$ an $\mathfrak m$-primary ideal with $J\subseteq I$ a minimal reduction, if the Ratliff-Rush filtration with respect to $I$ behaves well mod a superficial sequence $x_{1},x_{2},\ldots,x_{d-2} \in J$ for $I,$ then  $\widetilde{r}_{J}(I) \leq r_{J}(I).$ Indeed, set $x=x_{1},x_{2},\ldots,x_{d-2},$ $R'=R/(x),$ $ I'=I/(x)$ and $J'=J/(x),$ $\dim R'=2$ and $\widetilde{r}_{J}(I)=\widetilde{r}_{J'}(I').$ By \cite[Proposition 2.1]{ms}, we have $\widetilde{r}_{J'}(I') \leq r_{J'}(I') \leq r_{J}(I).$ Therefore, $\widetilde{r}_{J}(I) \leq r_{J}(I).$
 \end{enumerate}
     \end{remark}

\begin{definition} \label{mouma}
    Let $M$ be a graded $\mathcal{R}(I)$-module. We define 
        $\ex(M)=\sup\{n \in \mathbb{Z}: M_{n} \neq 0 \}.$
The $i$-th local cohomology module of $M$ with support in $\mathcal{R}(I)_{+}$ is denoted by  $\mathcal{H}^{i}_{\mathcal{R}(I)_{+}}(M).$ 
For the reference of local cohomology, we use \cite{bs}. If $M$ is a finitely generated $\mathcal{R}(I)$-module, then for each $i\geq 0$, $\left(\mathcal{H}^{i}_{\mathcal{R}(I)_{+}}(M)\right)_{n}=0$ for sufficiently large $n.$ Define $a_{i}(M)=\ex\left(\mathcal{H}^{i}_{\mathcal{R}(I)_{+}}(M)\right)$. The \textit{Castelnuovo-Mumford regularity} of $M$ is defined as
     \begin{equation*}
         \reg(M)=\max\{a_{i}(M)+i: 0 \leq i \leq \dim M\}.
     \end{equation*}

      Set $a_{i}^{*}(M)=\ex\left(\mathcal{H}^{i}_{\mathfrak M}(M)\right)$ and $\reg^{*}(M)=\max\{a_{i}^{*}(M)+i: 0 \leq i \leq \dim M\}$. Then from \cite[Remark 5.5]{tjp}, we get that for an $\mathfrak m$-primary ideal, $a_{i}^{*}(M)=a_{i}(M)$ and $\reg^{*}(M)=\reg(M)$.
\end{definition}

\begin{definition}\cite[5]{tjp} \label{def6}
The natural exact sequence of $R$-modules $$0\longrightarrow I^{n}/I^{n+1}\longrightarrow R/I^{n+1}\longrightarrow R/I^{n}\longrightarrow 0,$$ induces an exact sequence of $\mathcal{R}(I)$-modules called the \textit{first fundamental exact sequence}:
   \begin{equation} \label{gg}
       0\longrightarrow G(I) \longrightarrow L^{I}(R) \longrightarrow  L^{I}(R)(-1) \longrightarrow 0.
   \end{equation}
    The sequence in (\ref{gg}) induces the corresponding long exact sequence of local cohomolgy modules. For convenience from now onwards, we will write $\mathcal{H}^{i}_{\mathfrak M}\left(\_\right)=\mathcal{H}^{i}\left(\_\right)$ for all $i.$ 
    \begin{equation} \label{csk}
    \begin{aligned}
        0  &\longrightarrow \mathcal{H}^{0}\left(G(I)\right) \longrightarrow \mathcal{H}^{0}\left(L^{I}(R)\right) \longrightarrow \mathcal{H}^{0}\left(L^{I}(R)\right)(-1)  \\
     & \longrightarrow \mathcal{H}^{1}\left(G(I)\right)  \longrightarrow \mathcal{H}^{1}\left(L^{I}(R)\right)\longrightarrow \mathcal{H}^{1}\left(L^{I}(R)\right) (-1) \ldots
    \end{aligned}
\end{equation}

\end{definition}
\begin{definition}
Let $x$ be superficial for $I$,  set $\bar{R}=R/(x)$ and $\bar{I}=I/(x).$ We recall the 
\textit{second fundamental exact sequence} from  literature  \cite[6.2, 6.3]{tjp}. For each $n \geq 1,$ we have the following exact sequence of $R$-modules:
\begin{equation*}
    0\longrightarrow \frac{(I^{n+1}:x)}{I^{n}} \longrightarrow \frac{R}{I^{n}} \xlongrightarrow{\psi_{n}^{x}} \frac{R}{I^{n+1}} \longrightarrow \frac{\bar{R}}{\bar{I}^{n+1}}\longrightarrow 0,
\end{equation*}
where $\psi_{n}^{x}(a+I^{n})=xa+I^{n+1}$ for all  $a \in R.$ This sequence induces an exact sequence of $\mathcal{R}(I)$-modules, called the second fundamental exact sequence:
\begin{equation}
\label{eq:u}
    0\longrightarrow \mathcal{B}(x,R) \longrightarrow L^{I}(R)(-1
) \xlongrightarrow{\psi_{x}} L^{I}(R) \longrightarrow L^{\bar{I}}(\bar{R}) \longrightarrow 0,
\end{equation}
 where $\psi_{x}$ is multiplication by $x,$ and $\mathcal{B}(x,R)= \displaystyle \bigoplus_{n \geq 0} \frac{(I^{n+1}:x)}{I^{n}}.$
 Note that, $(I^{n+1}:x)=I^{n}$ for sufficiently large $n.$ Thus, $\mathcal{B}(x,R)$ has finite length. So, $\mathcal{H}^{0}\left(\mathcal{B}(x,R)\right)$ $=\mathcal{B}(x,R).$ The sequence in (\ref{eq:u}) induces a long exact sequence of local cohomology modules. 
\begin{equation} \label{4}
    \begin{aligned}
        0 \longrightarrow \mathcal{B}(x,R)  &  \longrightarrow \mathcal{H}^{0}\left(L^{I}(R)\right)(-1) \longrightarrow \mathcal{H}^{0}\left(L^{I}(R)\right) \longrightarrow \mathcal{H}^{0}\left(L^{\bar{I}}(\bar{R})\right) \\
     & \longrightarrow \mathcal{H}^{1}\left(L^{I}(R)\right)(-1) \longrightarrow \mathcal{H}^{1}\left(L^{I}(R)\right) \longrightarrow \mathcal{H}^{1}\left(L^{\bar{I}}(\bar{R})\right) \ldots
    \end{aligned}
\end{equation}
\end{definition}

For future references, we recall the following proposition:
\begin{proposition}\cite[Proposition 4.7]{trung} \label{trung001}
    Let $J=(x_{1},\ldots,x_{s})$ be a reduction of $I$ such that for a fixed integer $r \geq r_{J}(I)$, the sequence $x_{1},\ldots,x_{s}$ satisfies the following  condition:
    \begin{equation*}
        I^{r+1}\cap[(x_{1},\ldots,x_{i-1}):x_{i}]=(x_{1},\dots,x_{i-1})I^{r} \text{ for } i=1,\ldots,s.
    \end{equation*}Then for all $n \geq r$,  $I^{n+1}\cap[(x_{1},\ldots,x_{i-1}):x_{i}]=(x_{1},\dots,x_{i-1})I^{n}$ for $i=1,\ldots,s.$
\end{proposition}

\section{Bounds on the Stability Index for Cohen-Macaulay Local Rings}
 In this section, we give upper bounds on $\rho(I)$ in terms of reduction number with respect to a minimal reduction and higher Hilbert coefficients.  
  We begin this section with the following proposition,  which extends a key result  \cite[Proposition 4.3]{hjls}, by Huneke to $d$-dimensional Cohen-Macaulay local rings.

\begin{proposition} \label{maa}
     Let $(R,\mathfrak m)$ be a Cohen-Macaulay local ring of dimension $d \geq 2,$ $I$  an $\mathfrak m$-primary ideal and $J \subseteq I$ be a minimal reduction. If $\widetilde{I^{m}}=I^{m}$ for some $m \geq r_{J}(I)$, then $\widetilde{I^{n}}=I^{n}$ for all  $n \geq m.$ In particular, $\rho(I) \leq m.$
\end{proposition}

\begin{proof}
Using \cite[Lemma 1.2]{rtt}, any minimal reduction $J$ of $I$ can be generated by 
  $x_{1},x_{2},\ldots,x_{d},$ which forms a  superficial sequence for $I.$ Fix $m\geq r_{J}(I)$ such that $\widetilde{I^{m}}=I^{m}.$  We have $I^{m} \subseteq (I^{m+1}:x_{1}) \subseteq (\widetilde{I^{m+1}}:x_{1})$. By \cite[Lemma 3.1(5)]{rv}, we have $(\widetilde{I^{m+1}}:x_{1})=\widetilde{I^{m}},$ therefore, $(I^{m+1}:x_{1})=I^{m}$ as $\widetilde{I^{m}}=I^{m}$. 
 Since for every $n$, $I^{n+1} \cap (x_{1})=(x_{1})I^{n}$ if and only if $(I^{n+1}:x_{1})=I^{n},$ thus, we get
    \begin{equation} \label{rty}
        I^{m+1}\cap (x_{1})=(x_{1})I^{m}.
    \end{equation}
    Now, we will show that $I^{n+1}\cap[(x_{1}):x_{2}]=I^{n+1}\cap(x_{1})$ for all $n.$ Let $a \in I^{n+1}\cap[(x_{1}):x_{2}],$ then $ax_{2} \in (x_{1}).$ This implies that there exists $r \in R$ such that $ax_{2}+rx_{1}=0.$ Since $x_{1},x_{2}$ is a regular sequence in $R$, therefore, $a \in (x_{1}).$ As $a \in I^{n+1},$ we get $a \in I^{n+1}\cap(x_{1}).$ So we can conclude that $I^{n+1}\cap[(x_{1}):x_{2}] \subseteq I^{n+1}\cap(x_{1}).$ Since the other inclusion is obvious, therefore,
    \begin{equation} \label{qwe}
        I^{n+1}\cap[(x_{1}):x_{2}]=I^{n+1}\cap(x_{1}) \text{ for all } n.
    \end{equation}

    From equations (\ref{rty}) and (\ref{qwe}), we have $ I^{m+1}\cap[(x_{1}):x_{2}]=(x_{1})I^{m}.$ Then from Proposition \ref{trung001}, we get $ I^{n+1}\cap[(x_{1}):x_{2}]=(x_{1})I^{n}$ for all $n \geq m.$ This implies $I^{n+1}\cap(x_{1})=(x_{1})I^{n}$ for all $n \geq m,$ therefore, $(I^{n+1}:x_{1})=I^{n}$ for all $n \geq m.$ This implies $\rho_{x_{1}}(I) \leq m.$ Now from Definition \ref{phn}, we have $\rho(I) \leq m.$
\end{proof}

   \begin{remark} \label{shruti}
       In general, establishing a straightforward relationship between $r_{J}(I)$ and $\rho(I)$ is difficult to achieve. When  $\depth G(I) \geq 1,$  then $\widetilde{I^{n}}=I^{n}$ for all  $n \geq 1,$ therefore, $\rho(I) \leq r_{J}(I)$ for every minimal reduction $J$ of $I.$   However, there are examples which suggest that when $\depth G(I)=0,$ both inequalities are possible. Therefore, from Proposition \ref{maa}, we get that if  $(R, \mathfrak m)$ is Cohen-Macaulay of dimension $\geq 2$ and $\widetilde{I^{r_{J}(I)}} \neq I^{r_{J}(I)},$  then  $\widetilde{I^{n}} \neq I^{n}$ for all  $r_{J}(I) \leq n < \rho(I).$ In particular, $r_{J}(I) < \rho(I).$ 
       This also provides a lower bound on $\rho(I).$ Define $b:= \max \{r_{J}(I): J \text{ is a minimal reduction of }I\}.$ If $\widetilde{I^{b}} \neq I^{b},$ then $b < \rho(I).$ 
 \end{remark}

In the next theorem,  we give an upper bound on $\rho(I)$. 

\begin{theorem} \label{riyi}
    Let $R$ be a Cohen-Macaulay local ring of dimension  $d \geq 2$, $I$ an $\mathfrak m$-primary ideal and $J \subseteq I$ be a minimal reduction of $I$. 
    If $\widetilde{I^{r_{J}(I)}} \neq I^{r_{J}(I)},$ then 
    \begin{equation*}
           \rho(I) \leq r_{J}(I)-1+(-1)^{d+1}(e_{d+1}(I)-\widetilde{e}_{d+1}(I)).
    \end{equation*}
\end{theorem}
\begin{proof}
   Since $\widetilde{I^{r_{J}(I)}} \neq I^{r_{J}(I)},$ from Remark \ref{shruti}, we have $\widetilde{I^{n}} \neq I^{n}$ for all  $r_{J}(I) \leq n < \rho(I).$ Thus, we get
    \begin{equation*}
    \begin{split}
         \rho(I) & \leq r_{J}(I)-1+\sum_{n=r_{J}(I)-1}^{\rho(I)-2}\lambda(\widetilde{I^{n+1}}/I^{n+1})\\
         & \leq r_{J}(I)-1+\sum_{
         n=0}^{m}\lambda(\widetilde{I^{n+1}}/I^{n+1}) \text{ (where $m \gg 0$ is an integer)}\\
        & \text{$= r_{J}(I)-1+\sum_{n=0}^{m}\lambda(R/I^{n+1})-\sum_{n=0}^{m}\lambda(R/\widetilde{I^{n+1}})$}\\
        & = r_{J}(I)-1 + H_{I}^{2}(m)-H_{\mathcal{F}}^{2}(m) \text{ (where $\mathcal{F}:=\{\widetilde{I^{n}}\}_{n \geq 0}$)}\\
        & = \text{$r_{J}(I)-1+P_{I}^{2}(m)-P_{\mathcal{F}}^{2}(m)$ (since $m\gg0$)}\\
        & = r_{J}(I)-1+(-1)^{d+1}(e_{d+1}(I)-\widetilde{e}_{d+1}(I)). 
    \end{split}
        \end{equation*}
        Note that the last equality holds since $e_{i}(I)=\widetilde{e}_{i}(I)$ for all $0 \leq i\leq  d.$ \end{proof}

         As an easy consequence, we have the following corollary under the hypothesis that the Ratliff-Rush filtration behaves well mod a superficial sequence.
         
\begin{corollary} \label{nautanki}
 Let $R$ be a Cohen-Macaulay local ring of odd dimension $d\geq 3$, $I$ an $\mathfrak m$-primary ideal and $J \subseteq I$ be a minimal reduction of $I$. Suppose the Ratliff-Rush filtration with respect to $I$ behaves well mod a superficial sequence $x_{1}, \ldots, x_{d-2}  \in I.$  If $\widetilde{I^{r_{J}(I)}} \neq I^{r_{J}(I)},$  then \begin{equation*}
     \rho(I) \leq r_{J}(I)-1+e_{d+1}(I).
 \end{equation*}
\end{corollary}
\begin{proof}
Since the Ratliff-Rush filtration with respect to $I$ behaves well mod a superficial sequence $x_{1},  \ldots, x_{d-2}  \in I,$  from Remark \ref{21} $(i)$, we have $\depth \widetilde{G}(I) \geq d-1.$ Thus, from \cite[Theorem 2.5]{rv}, we have $\widetilde{e}_{d+1}(I)=\displaystyle \sum_{n\geq d}\binom{n}{d}\lambda(\widetilde{I^{n+1}}/J\widetilde{I^{n}}) \geq 0.$ Therefore,  $\rho(I) \leq r_{J}(I)-1+e_{d+1}(I).$ 
\end{proof}

Note that the bound in Corollary \ref{nautanki} gives a linear bound on $\rho(I).$ In the next corollary, we give an upper bound on $\rho(I)$ for Cohen-Macaulay local rings of even dimension. Before we proceed, we need the following lemma. We remark that the technique used in the proof of the following lemma is inspired from \cite[Proposition 3.5]{sy}.

\begin{lemma} \label{prativa}
    Let $(R, \mathfrak m)$ be a Cohen-Macaulay local ring of 
 even dimension $d\geq2$, $I$ an $\mathfrak m$-primary ideal and $J \subseteq I$ be a minimal reduction of $I.$ Suppose that the Ratliff-Rush filtration with respect to $I$ behaves well mod a superficial sequence $x_{1}, \ldots, x_{d-2} \in I$,  then
    \begin{equation*}
        \widetilde{e}_{d+1}(I) \leq \left(\frac{r_{J}(I)-1}{d}\right)(e_{d}(I)-e_{d-1}(I)+\cdots+e(I)-\lambda(R/I)+\lambda(\widetilde{I}/I)). 
    \end{equation*}
\end{lemma}

\begin{proof} Set $\widetilde{r}_{J}(I)=\widetilde{r}$ and $v(n)=\lambda(\widetilde{I^{n+1}}/J\widetilde{I^n}).$ Note that $v(n)=0$ for all $n \geq \widetilde{r}.$ Since the Ratliff-Rush filtration with respect to $I$ behaves well mod a superficial sequence $x_{1},  \ldots, x_{d-2}  \in I,$ therefore, from Remark \ref{21} $(i)$, we have $\depth \widetilde{G}(I) \geq d-1.$ Thus, from \cite[Theorem 2.5]{rv}, we have $\widetilde{e}_{i}(I)=\displaystyle \sum_{n\geq i-1}\binom{n}{i-1}v(n),$ for all $i \geq 1.$ Consider
   
         $\widetilde{e}_{d+1}(I) - \left(\frac{\widetilde{r}-1}{d}\right)(e_{d}(I)-e_{d-1}(I)+\cdots+e(I)-\lambda(R/I))$
                     \small{ 
          \begin{equation*}
            \begin{split}
            &= \displaystyle \sum_{n = d}^{\widetilde{r}-1} \binom{n}{d}v(n)- \left(\frac{\widetilde{r}-1}{d}\right)\left(\displaystyle \sum_{n = d-1}^{\widetilde{r}-1}\binom{n}{d-1} v(n)-  \displaystyle \sum_{n = d-2}^{\widetilde{r}-1}\binom{n}{d-2} v(n)+\cdots+\lambda(R/J)-\lambda(R/I)\right) \\
             &= \displaystyle \sum_{n = 1}^{\widetilde{r}-1} \binom{n}{d}v(n)- \left(\frac{\widetilde{r}-1}{d}\right)\left(\displaystyle \sum_{n = 1}^{\widetilde{r}-1}\binom{n}{d-1} v(n)-  \displaystyle \sum_{n = 1}^{\widetilde{r}-1}\binom{n}{d-2} v(n)+\cdots+\lambda(R/J)-\lambda(R/I)\right)\\
            &=\displaystyle \sum_{n = 1}^{\widetilde{r}-1} \binom{n}{d}v(n)- \left(\frac{\widetilde{r}-1}{d}\right)\left(\displaystyle \sum_{n = 1}^{\widetilde{r}-1}\left(\displaystyle \sum_{j=0}^{d-1}(-1)^{d-1-j}\binom{n}{j} \right)v(n)-\lambda(\widetilde{I}/J)+\lambda(R/J)-\lambda(R/I)\right)\\
             &=\displaystyle \sum_{n = 1}^{\widetilde{r}-1} \binom{n}{d}v(n)- \left(\frac{\widetilde{r}-1}{d}\right)\left(\displaystyle \sum_{n = 1}^{\widetilde{r}-1}\binom{n-1}{d-1}v(n)-\lambda(\widetilde{I}/I)\right) \text{ (since $\displaystyle \sum_{j=0}^{d-1}(-1)^{d-1-j}\binom{n}{j}=\binom{n-1}{d-1}$ for every $d \geq 1$)}\\
             &=\displaystyle \sum_{n = 1}^{\widetilde{r}-1}\left(\frac{n(n-1)(n-2)\ldots(n-d+1)}{d!}- \left(\frac{\widetilde{r}-1}{d}\right)\frac{(n-1)(n-2)\ldots(n-d+1)}{(d-1)!}\right)v(n)+\left(\frac{\widetilde{r}-1}{d}\right)\lambda(\widetilde{I}/I)\\
             &=\displaystyle \sum_{n = 1}^{\widetilde{r}-1}\left(\frac{(n-1)(n-2)\ldots(n-d+1)}{d!} (n-\widetilde{r}+1)\right)v(n)+\left(\frac{\widetilde{r}-1}{d}\right)\lambda(\widetilde{I}/I)\\
              &\leq \left(\frac{\widetilde{r}-1}{d}\right)\lambda(\widetilde{I}/I),
        \end{split}
       \end{equation*}
       }
            
            \normalsize{since $(n-\widetilde{r}+1) \leq 0$ for all $1 \leq n \leq \widetilde{r}-1.$ Therefore, $\widetilde{e}_{d+1}(I) \leq \left(\frac{\widetilde{r}_{J}(I)-1}{d}\right)(e_{d}(I)-e_{d-1}(I)+\cdots+e_{0}(I)-\lambda(R/I)+\lambda(\widetilde{I}/I)). $ From Remark \ref{21} $(iii)$, we have  $\widetilde{e}_{d+1}(I) \leq \left(\frac{r_{J}(I)-1}{d}\right)(e_{d}(I)-e_{d-1}(I)+\cdots+e(I)-\lambda(R/I)+\lambda(\widetilde{I}/I))$}.
\end{proof}

\begin{corollary} \label{komal}
     Let $(R,\mathfrak m)$ be a Cohen-Macaulay ring of even dimension $d\geq2$ and $I$ an $\mathfrak m$-primary ideal. Suppose the Ratliff-Rush filtration  with respect to $I$ behaves well mod a superficial sequence $x_{1},  \ldots, x_{d-2} \in I$. For a minimal reduction $J$ of $I$, if $r_{J}(I) < \rho(I),$ then  
     \begin{equation*}
        \rho(I) \leq r_{J}(I)-1-e_{d+1}(I)+\left(\frac{r_{J}(I)-1}{d}\right)(e_{d}(I)-e_{d-1}(I)+\cdots+e(I)-\lambda(R/I)+\lambda(\widetilde{I}/I)).
    \end{equation*}
\end{corollary}

\begin{proof}    From Theorem \ref{riyi}, we have $\rho(I) \leq r_{J}(I)-1+\widetilde{e}_{d+1}(I)-e_{d+1}(I).$ Again, from Lemma \ref{prativa}, we have $\widetilde{e}_{d+1}(I) \leq \left(\frac{r_{J}(I)-1}{d}\right)(e_{d}(I)-e_{d-1}(I)+\cdots+e_{0}(I)-\lambda(R/I)+\lambda(\widetilde{I}/I))$. Therefore, $\rho(I) \leq r_{J}(I)-1-e_{d+1}(I)+\left(\frac{r_{J}(I)-1}{d}\right)(e_{d}(I)-e_{d-1}(I)+\cdots+e(I)-\lambda(R/I)+\lambda(\widetilde{I}/I)).$
    \end{proof}

We now establish bounds on the stability index of the Ratliff-Rush filtration, with a particular focus on dimension two.

    \begin{proposition}\label{mushroom}
         Let $(R, \mathfrak m)$ be a two dimensional Cohen-Macaulay local ring, $I$ an $\mathfrak m$-primary ideal and $J \subseteq I$ be a minimal reduction of $I.$ If  $r_{J}(I) < \rho(I),$ then
        \begin{equation*}
            \rho(I) \leq r_{J}(I)-1-e_{3}(I)+\left(\frac{r_{J}(I)-1}{2}\right)(e_{2}(I)-e_{1}(I)+e(I)-\lambda(R/I)+\lambda(\widetilde{I}/I)).
        \end{equation*}
   \end{proposition}

   \begin{proof}
       From Corollary \ref{komal}, for $d=2$, we get  $\rho(I) \leq r_{J}(I)-1-e_{3}(I)+\left(\frac{r_{J}(I)-1}{2}\right)(e_{2}(I)-e_{1}(I)+e(I)-\lambda(R/I)+\lambda(\widetilde{I}/I)).$
   \end{proof}

    In the next proposition, we give another bound on $\rho(I)$. Note that the bound given in Proposition \ref{mushroom} depends on the Ratliff-Rush closure of $I$, which can be challenging to compute. In contrast, the following bound is computationally more accessible as it relies solely on the higher Hilbert coefficients. Before proceeding further, it is necessary to introduce the following lemma. We note that the proof follows a similar approach to \cite[Theorem 4.1]{ms}, but we include the steps for completeness.

    \begin{lemma} \label{mainak}
            Let $(R, \mathfrak m)$ be a two dimensional Cohen-Macaulay local ring, $I$ an $\mathfrak m$-primary ideal. Then $\widetilde{e}_{3}(I) \leq (e_{2}(I)-1)e_{2}(I).$ 
    \end{lemma}

    \begin{proof} Let $\mathcal{F}:=\{\widetilde{I^{n}}\}_{n \geq 0}$ and $H_{\mathcal{F}}(n):=\lambda(R/\widetilde{I^{n}})$ for all $n \in \mathbb Z$. By the difference formula in \cite[Proposition 4.4]{blancafort}, we have for all $n \geq -1$
    \begin{equation} \label{alok}
        P_{\mathcal{F}}(n)-H_{\mathcal{F}}(n)=\lambda(\mathcal{H}^{2}_{\mathcal{R_{+}}}(\mathcal{R}(\mathcal{F}))_{n+1}).
    \end{equation}
    Now taking sum for large $m$ on both sides of the equation, we get
        \begin{align} \label{ee} \nonumber
            \displaystyle\sum_{n=0}^{m}\lambda(\mathcal{H}^{2}_{\mathcal{R_{+}}}(\mathcal{R}(\mathcal{F}))_{n+1})&= \displaystyle\sum_{n=0}^{m}P_{\mathcal{F}}(n)-\displaystyle\sum_{n=0}^{m}H_{\mathcal{F}}(n)\\ \nonumber
            & = \displaystyle\sum_{n=0}^{m}P_{\mathcal{F}}(n)-H_{\mathcal{F}}^{2}(m)\\ \nonumber
            & = \widetilde{e}_{0}(I)\binom{m+3}{3}-\widetilde{e}_{1}(I)\binom{m+2}{2}+\widetilde{e}_{2}(I)\binom{m+1}{1}-P_{\mathcal{F}}^{2}(m)\\
            &=\widetilde{e}_{3}(I).
        \end{align}
    
        Note that $R$ is a two-dimensional Cohen-Macaulay ring,  by \cite[Lemma 4.7]{blancafort}, we have $$\lambda(\mathcal{H}^{2}_{\mathcal{R_{+}}}(\mathcal{R}(\mathcal{F}))_{n}) \leq \lambda(\mathcal{H}^{2}_{\mathcal{R_{+}}}(\mathcal{R}(\mathcal{F}))_{n-1}) \text{ for all }  n \in \mathbb Z.$$ Now in equation (\ref{alok}), we substitute $n=-1$ to get
        \begin{equation} \label{sl}
            \lambda(\mathcal{H}^{2}_{\mathcal{R_{+}}}(\mathcal{R}(\mathcal{F}))_{0})=\widetilde{e}_{2}(I)=e_{2}(I).
        \end{equation}
        Therefore, combining equations (\ref{ee}) and (\ref{sl}), we get 
        \begin{equation} \label{nochance}
            \widetilde{e}_{3}(I)=\displaystyle\sum_{n=0}^{m}\lambda(\mathcal{H}^{2}_{\mathcal{R_{+}}}(\mathcal{R}(\mathcal{F}))_{n+1}) \leq \displaystyle\sum_{n=0}^{a_{2}(\mathcal{R(\mathcal{F})})-1}\lambda(\mathcal{H}^{2}_{\mathcal{R_{+}}}(\mathcal{R}(\mathcal{F}))_{0})=a_{2}(\mathcal{R(\mathcal{F})})e_{2}(I),
        \end{equation}
        where $a_{2}(\mathcal{R(\mathcal{F})}) \leq a_{2}(\widetilde{G}(I))$ from the proof method of \cite[Theorem 3.1]{trung} for the filtration
$\mathcal{F} := \{ \widetilde{I_{^n}} \}_{n \ge 0}.$  Furthermore, by \cite[Corollary 5.7(2)]{tjm}, we have $a_{2}(\widetilde{G}(I)) = \widetilde{r}(I)-2$ and by \cite[Theorem 2.2]{ms} $\widetilde{r}(I)-1 \leq e_{2}(I).$   Therefore, by equation (\ref{nochance}), we have $\widetilde{e}_{3}(I) \leq (e_{2}(I)-1)e_{2}(I).$ 
    \end{proof}

\begin{proposition} \label{riya}
    Let $(R,\mathfrak m)$ be a  two dimensional Cohen-Macaulay local ring  and $I$ an $\mathfrak m$-primary ideal. For a minimal reduction $J$ of $I$, if $r_{J}(I) < \rho(I),$ then  
    \begin{equation*}
        \rho(I) \leq r_{J}(I)-1+(e_{2}(I)-1)e_{2}(I)-e_{3}(I).
    \end{equation*}
\end{proposition}

\begin{proof}
    From Theorem \ref{riyi}, we have $\rho(I) \leq r_{J}(I)-1+\widetilde{e}_{3}(I)-e_{3}(I).$ Now, from Lemma \ref{mainak}, we have $\widetilde{e}_{3}(I) \leq (e_{2}(I)-1)e_{2}(I).$ Therefore, $\rho(I) \leq  r_{J}(I)-1+(e_{2}(I)-1)e_{2}(I)-e_{3}(I).$
\end{proof}

Note that from  the proof of Theorem \ref{riyi}, we have $e_{3}(I)=\widetilde{e}_{3}(I) - \displaystyle \sum_{n \geq 0}\lambda \left(\widetilde{I^{n+1}}/I^{n+1}\right) \leq \widetilde{e}_{3}(I) $   
and from Lemma \ref{mainak}, we have $\widetilde{e}_{3}(I) \leq (e_{2}(I)-1)e_{2}(I).$ Therefore, $e_{3}(I) \leq (e_{2}(I)-1)e _{2}(I).$ This implies that if $e_{2}(I) \leq 1$, then $e_{3}(I) \leq 0,$ which is  similar to a relation between  $e_{2}(I)$ and $e_{3}(I)$ that was proved in \cite[Proposition 6.4]{tjp2} and \cite[Proposition 2.6]{mafi2}, in a three-dimensional Cohen-Macaulay ring.
In the next  corollary, we observe that if $e_{3}(I)$ attains its upper bound, then $\rho(I)$ is always less than $r_{J}(I)$ for every minimal reduction $J$ of $I.$

  \begin{corollary}
          Let $(R, \mathfrak m)$ be a two dimensional Cohen-Macaulay local ring, $I$ an $\mathfrak m$-primary ideal. If $e_{3}(I)= (e_{2}(I)-1)e_{2}(I),$ then $\rho(I) \leq r_{J}(I)$ for every minimal reduction $J$ of $I.$
    \end{corollary}

    \begin{proof}
  Suppose for a minimal reduction $J$ of $I$, $r_{J}(I) < \rho(I),$ then by Proposition \ref{riya}, we have $\rho(I) \leq r_{J}(I)-1,$ a contradiction. Therefore, if $e_{3}(I)= (e_{2}(I)-1)e_{2}(I),$ then  $\rho(I) \leq r_{J}(I)$ for every minimal reduction $J$ of $I.$ 
 \end{proof}

The following example illustrates the above propositions.

\begin{example}  \cite[Example 4.6]{rtt}
    Let $R=\mathbb Q[[x,y]]$ and $I=(x^7,x^6y,x^2y^5,y^7)$. By CoCoA,  the Hilbert series is \begin{equation*}
        h_{I}(t)=\frac{35+6t+7t^2+2t^3-t^5}{(1-t)^2}.
    \end{equation*}
    
Then $e(I)=49$, $e_{1}(I)=21$, $e_{2}(I)=3$ and $e_{3}(I)=-8$. By \cite[Theorem 2.1]{elias}, we have $\widetilde{I}=(x^7,x^6y,x^4y^3,x^2y^5,y^7)$ and $\lambda(\widetilde{I}/I)=4.$ Here $J=(x^7, x^6y+y^7)$ is a minimal reduction of $I$ with $r_{J}(I)=3$. Using Macaulay 2, we get $x^{17}y^{4}\in (I^{4}:I) \subseteq \widetilde{I^{3}}$ but $x^{17}y^{4} \notin I^{3}$, therefore, $\widetilde{I^{3}} \neq I^{3}$.  Thus, from Proposition \ref{mushroom}, we get $\rho(I) \leq r_{J}(I)-1-e_{3}(I)+\left(\frac{r_{J}(I)-1}{2}\right)(e_{2}(I)-e_{1}(I)+e(I)-\lambda(R/I)+\lambda(\widetilde{I}/I))=10$. Further, from Proposition \ref{riya}, we get $\rho(I)\leq r_{J}(I)-1+(e_{2}(I)-1)e_{2}(I)-e_{3}(I)=16$.
\end{example}

In the next corollary, we establish a bound on $\rho(I)$ in terms of Hilbert coefficients using Rossi's bound on the reduction number.

 \begin{corollary}\label{swara}
        Let $(R,\mathfrak m)$ be a two dimensional Cohen-Macaulay local ring and $I$ an $\mathfrak m$-primary ideal.  For a minimal reduction $J$ of $I$, if $r_{J}(I) < \rho(I),$ then  
    \begin{equation*}
        \rho(I) \leq e_{1}(I)-e_{0}(I)+\lambda(R/I)+(e_{2}(I)-1)e_{2}(I)-e_{3}(I).
    \end{equation*}
     \end{corollary}

\begin{proof}
    From \cite[Corollary 1.5]{rossi2} we have, $r_{J}(I) \leq e_{1}(I)-e_{0}(I)+\lambda(R/I)+1.$ Thus,  the conclusion follows from Proposition \ref{riya}.
\end{proof}

In \cite[Question 3.14]{cmn}, Miranda-Neto and Queiroz proposed the following question: In a two-dimensional Buchsbaum ring $(R,\mathfrak m)$ with positive depth, is it true that $\rho(I) \leq r_{J}(I)+1,$ for an $\mathfrak m$-primary ideal $I$ with minimal reduction $J?$
In the next proposition, we  give a class of $\mathfrak m$-primary ideals for which $\rho(I)$ is at most $r_{J}(I)+1.$ This gives a partial answer to the above question for Cohen-Macaulay local rings. 

\begin{proposition} \label{barkhanda}
    Let $(R, \mathfrak m)$ be a two dimensional Cohen-Macaulay local ring, $I$ an $\mathfrak m$-primary ideal such that $e_{2}(I)=0$ or $1$ and $e_{3}(I)=-1.$ Then for any minimal reduction $J \subseteq I$, 
    \begin{equation*}
        \rho(I) \leq r_{J}(I)+1.
    \end{equation*}
\end{proposition}

\begin{proof}
    Let $e_{2}(I)=0$ or $1.$ From \cite[Theorem 2.5]{rv}, we have $\widetilde{e_{2}}(I)= \displaystyle \sum_{n \geq 1}n \lambda \left(\widetilde{I^{n+1}}/J\widetilde{I^{n}}\right).$ Also, $\widetilde{e_{2}}(I)=e_{2}(I)=0$ or $1$, therefore,   $\widetilde{I^{n+1}} = J\widetilde{I^{n}}$ for all $n \geq 2.$  Again from \cite[Theorem 2.5]{rv}, we have $\widetilde{e_{3}}(I)= \displaystyle \sum_{n \geq 2} \binom{n}{2} \lambda \left(\widetilde{I^{n+1}}/J\widetilde{I^{n}}\right),$ thus $\widetilde{e_{3}}(I)=0,$ this implies  $e_{3}(I)=-\displaystyle\sum_{n\geq 1}\lambda(\widetilde{I^{n}}/I^{n}).$ Since $e_{3}(I)=-1$, therefore, $\displaystyle \sum_{n\geq 1}\lambda(\widetilde{I^{n}}/I^{n})=1.$ Note that $\widetilde{I^{n}}=I^{n}$ for all $n \geq \rho(I)$ and $\widetilde{I^{\rho(I)-1}}\neq I^{\rho(I)-1}$.  Thus, $\displaystyle \sum_{n=1}^{n=\rho(I)-1}\lambda(\widetilde{I^{n}}/I^{n})=1,$ this implies $\widetilde{I^{m}}=I^{m}$ for all $1 \leq m \leq \rho(I)-2.$  Suppose $r_{J}(I) < \rho(I)-1,$ then $\widetilde{I^{r_{J}(I)}}=I^{r_{J}(I)},$ thus by \cite[Proposition 4.3]{hjls} $\widetilde{I^{n}}=I^{n}$ for all $n \geq r_{J}(I)$, which implies $\widetilde{I^{\rho(I)-1}}=I^{\rho(I)-1},$  a contradiction. Therefore, $\rho(I) \leq r_{J}(I)+1.$ 
\end{proof}

The following example illustrates the above proposition. 

\begin{example} \cite[Example 3.3]{rv}
     Let $R=\mathbb Q[[x,y]]$ and $I=(x^4,x^3y,xy^3,y^4).$ By CoCoA, the Hilbert series is  \begin{equation*}
         h_{I}(t)=\frac{11 + 3t + 3t^2 - t^3}{(1-t)^2}.
     \end{equation*}
     
   Then $e_{2}(I)=0$ and $e_{3}(I)=-1$. Note that $J=(x^4,y^4)$ is a minimal reduction of $I$ and $r_{J}(I)=2$. Therefore, from \cite[Proposition 4.3]{hjls}, we have $\widetilde{I^{2}}=I^2$, hence, $\rho(I)=2.$ Thus,  from Proposition \ref{barkhanda}, we have $\rho(I) \leq r_{J}(I)+1=3.$  
\end{example}

In the next proposition, we give another bound on $\rho(I)$ without  any assumptions on the Hilbert coefficients.

\begin{proposition} \label{shoni}
    Let $(R, \mathfrak m)$ be a two dimensional Cohen-Macaulay local ring, and $I$ an $\mathfrak m$-primary ideal with $\depth G(I^t)>0$ for some $t> 1.$ Let $J$ be a minimal reduction of $I$ such that $\widetilde{I^{r_{J}(I)}} \neq I^{r_{J}(I)}$.  If $\rho(I) \equiv 0 \mod t$ then 
   $\rho(I) \leq r_{J}(I)-1+t$.
   Furthermore, if $\rho(I) \equiv k \mod t$ for $1 \leq k \leq t-1,$ then
     $\rho(I) \leq r_{J}(I)+k$.
\end{proposition}

\begin{proof}
Suppose $\rho(I)=mt$ for some $m > 1.$ Then by Remark \ref{shruti}, $(m-1)t< r_{J}(I)<mt,$ this implies $mt-t <r_{J}(I).$ Thus, $\rho(I) < r_{J}(I)+t.$  Therefore, $\rho(I) \leq r_{J}(I)-1+t$.

  Now, suppose $\rho(I) = mt+k$ for some $m \geq 1.$ We claim that $\rho(I) \leq r_{J}(I)+k.$ Suppose $r_{J}(I) < \rho(I)-k=mt$, then $r_{J}(I) < mt < \rho(I)$. Therefore, by Remark \ref{shruti}, $\widetilde{I^{mt}} \neq I^{mt},$ but $\depth G(I^t)>0 $ which implies $\widetilde{I^{mt}} = I^{mt}$ for all $m$. This is a contradiction. Therefore,
         $\rho(I) \leq r_{J}(I)+k$.
   \end{proof}

As a immediate corollary we get 

\begin{corollary}\label{priya}
    Let $(R, \mathfrak m)$ be a two dimensional Cohen-Macaulay local ring, and $I$ an $\mathfrak m$-primary ideal.  Let $J$ be a minimal reduction of $I.$ Suppose $\depth G(I^2)>0$  and $r_{J}(I)$ is odd.  If $\widetilde{I^{r_{J}(I)}} \neq I^{r_{J}(I)}$ then 
$\rho(I) = r_{J}(I)+1.$
\end{corollary}

\begin{proof}
   Note that by Remark \ref{shruti}, $\rho(I)=2m$ for some $m>1.$ By Proposition \ref{shoni}, for $t=2,$ $\rho(I) \leq r_{J}(I)+1.$ Since $\widetilde{I^{r_{J}(I)}} \neq I^{r_{J}(I)},$ therefore, $r_{J}(I) < \rho(I) \leq r_{J}(I)+1,$ which implies $\rho(I)=r_{J}(I)+1.$ 
\end{proof}

In \cite{ankit}, the authors proved that if $R$ is Cohen-Macaulay with $e_{2}(\mathfrak m)=e_{1}(\mathfrak m)-e(\mathfrak m)+1\neq0$, then $\type(R) \geq e(\mathfrak m)-\mu(\mathfrak m)+d-1$ where $\type(R)=\dim_{k} \Ext_{R}^{d}(k,R)$ and $\mu(\mathfrak m)= \lambda(\mathfrak m/\mathfrak m^{2}).$ Further, if $e_{2}(\mathfrak m)=e_{1}(\mathfrak m)-e(\mathfrak m)+1\neq0$ and $\type(R) = e(\mathfrak m)-\mu(\mathfrak m)+d-1$, then $G(\mathfrak m)$ is Cohen-Macaulay. In the next corollary, we consider the next boundary case i.e., $\type(R) = e(\mathfrak m)-\mu(\mathfrak m)+d.$

\begin{corollary}
    Let $(R,\mathfrak m)$ be a two dimensional Cohen-Macaulay local ring with $e_{2}(\mathfrak m)=e_{1}(\mathfrak m)-e(\mathfrak m)+1\neq0$ and $\type R=e(\mathfrak m)-\mu(\mathfrak m)+2.$ Suppose $J$ is a minimal reduction of $\mathfrak m$ such that $\widetilde{\mathfrak m^{r_{J}(I)}} \neq \mathfrak m^{r_{J}(I)},$ then $\rho(\mathfrak m) \leq r_{J}(\mathfrak m)+3.$
\end{corollary}
\begin{proof}
    If $\depth G(\mathfrak m) \geq 1,$ then $\rho(\mathfrak m)=1 \leq r_{J}(\mathfrak m)+3.$ Suppose $\depth G(\mathfrak m)=0.$ By \cite[Theorem 4.1]{ankit}, $\widetilde{\mathfrak m^{i}}=\mathfrak m^{i}$ for all $i \geq 3,$ which implies $\depth G(\mathfrak m^{3}) >0.$   Therefore, by Proposition \ref{shoni}, if   $\rho(I) \equiv 0 \mod 3$ then 
   $\rho(I) \leq r_{J}(I)+2 < r_{J}(I)+3$ and 
 if $\rho(I) \equiv k \mod t$ for $1 \leq k \leq 2,$ then
     $\rho(I) \leq r_{J}(I)+3$. 
\end{proof}

We conclude this section with the following  proposition, in which we examine the relationship between $r_{J}(I)$ and $\rho(I)$, focusing specifically on the cases where $\rho(I)$ takes small values. This proposition highlights the condition under which $\rho(I) \leq r_{J}(I).$

\begin{proposition}\label{archita}
   Let $(R,\mathfrak m)$ be a two dimensional Cohen-Macaulay local ring, $I$ an $\mathfrak m$-primary ideal with minimal reduction $J$. Then the following holds:
    \begin{enumerate}[(i)]
        \item If $\rho(I) =2,$ then $\rho(I) \leq r_{J}(I).$
        \item If $\rho(I) =3$ and $I$ is Ratliff-Rush closed, then $\rho(I) \leq r_{J}(I).$
    \end{enumerate}
\end{proposition}

\begin{proof}
    \begin{enumerate}[(i)]
        \item Suppose $\rho(I)=2,$ then from \cite[Remark 1.6]{rs}, we have $\depth G(I)=0$. Let $x \in J$ be a superficial element for $I.$ If $\rho(I/(x))=1,$ then again by \cite[Remark 1.6]{rs}, $\depth G(I/(x))>0$. Hence, by Sally's Descent, $\depth G(I)>1$, which is a contradiction. Therefore, $\rho(I) \leq \rho(I/(x)) \leq r_{J/(x)}(I/(x))\leq r_{J}(I),$ where the second inequality follows from \cite[Proposition 4.2$(i)$]{rs}.
 \item Suppose $\rho(I)=3$. Let $x \in J$ be a superficial element for $I,$ then from the similar arguments as above, we have $\rho(I/(x)) > 1.$ Note that $\frac{I^{n}+(x)}{(x)} \subseteq \frac{\widetilde{I^{n}}+(x)}{(x)} \subseteq \widetilde{\frac{I^{n}+(x)}{(x)}}$ for all $n \geq 1.$ If possible let $\rho(I/(x)) = 2,$ then $\widetilde{\frac{I^{n}+(x)}{(x)}}=\frac{I^{n}+(x)}{(x)}=\frac{\widetilde{I^{n}}+(x)}{(x)}$ for all $n \geq 2.$ At $n=2,$ $\widetilde{I^{n}} \neq I^{n}$. Therefore, there exist $y \in \widetilde{I^{2}}\backslash I^{2},$ such that $\Bar{y} \in \frac{\widetilde{I^{2}}+(x)}{(x)}=\frac{I^{2}+(x)}{(x)}.$ This implies $y \in I^{2}+(x),$ thus there exists $z \in I^{2}$ such that $y=z+rx$ for some $r \in R.$ Note that $y-z=rx \in \widetilde{I^{2}},$ this implies $r \in (\widetilde{I^{2}}:x)=\widetilde{I}$. Since $I$ is Ratliff-Rush closed, therefore $r \in I$ which implies $rx \in I^{2}.$ Thus, $y \in I^{2},$ which is a contradiction. Therefore, $\rho(I/(x)) >2.$ Thus, $\rho(I) \leq \rho(I/(x)) \leq r_{J/(x)}(I/(x))\leq r_{J}(I),$ where the second inequality follows from \cite[Proposition 4.2$(i)$]{rs}.  \qedhere
    \end{enumerate}
\end{proof}

\section{Bounds on the Stability Index for Buchsbaum Rings}
We begin this section by extending \cite[Proposition 4.3]{hjls} to two-dimensional Buchsbaum rings.

 \begin{lemma} \label{papa}
     Let $R$ be a two dimensional Buchsbaum ring with $\depth R>0$ and $I$ an $\mathfrak m$-primary ideal. Let $x,y$ be a superficial sequence for $I$ such that $J=(x,y)$ is a minimal reduction of $I.$  Set $r_{J}(I)=r.$ If $\widetilde{I^{m}}=I^{m}$ for some $m \geq r,$ then $\widetilde{I^{n}}=I^{n}$ for all $n \geq m.$  In particular, $\rho(I) \leq m.$
 \end{lemma}
\begin{proof}
 We will first show that $(I^{n+1}:x)=I^{n}+y(I^{n}:x)$ for all $n \geq r+1.$  Let $f$ be an arbitrary element in $(I^{n+1}:x)$ for any $n\geq r+1,$ this implies that $xf\in I^{n+1}.$ Since $I^{n+1}=(x,y)I^{n},$ thus there are elements $p,q \in I^{n}$ such that $xf=xp+yq,$ this implies $q \in [(x):y].$ From this, it follows that $q \in I^{n}\cap [(x):y].$ Now from the proof of \cite[Theorem 2.4]{rtt}, we have
     \begin{equation*}
         I^{n}\cap [(x):y]=I^{n}\cap (x)=x(I^{n}:x).
     \end{equation*}
     Hence, $q=xq'$ for some $q' \in (I^{n}:x).$ Thus, $xf=xp+xyq'$. Since $x$ is a non-zero divisor, we have $f=p+yq'\in I^{n}+y(I^{n}:x).$ Therefore, $(I^{n}:x) \subseteq I^{n}+y(I^{n}:x).$ The reverse inclusion is obvious, thus we have $(I^{n+1}:x)=I^{n}+y(I^{n}:x)$ for all $n \geq r+1.$
     
   Now it is enough to prove that if $\widetilde{I^{n}}=I^{n}$ for $n=m+1$, then the conclusion follows.  
   By \cite[Lemma 3.1 (5)]{rv}, we get the first equality in the chain
   $$I^{m} \subseteq (I^{m+1}:x) \subseteq (\widetilde{I^{m+1}}:x)= \widetilde{I^{m}}=I^{m}.$$
   The last equality follows from the hypothesis. Hence equalities hold from left to right in the last chain; in particular,  $(I^{m+1}:x)=I^{m}.$ Next as $m+1 \geq r+1,$ by the first paragraph of the proof 
   $$(I^{m+2}:x)=I^{m+1}+y(I^{m+1}:x)=I^{m+1}+yI^{m}=I^{m+1}.$$
   Arguing similarly, we have $(I^{n+1}:x)=I^{n}$ for all $n\geq m.$ Thanks to  Definition \ref{phn}, we deduce $\rho(I)=\rho_{x}(I) \leq m,$ as desired.
 \end{proof}

The formula in Proposition \ref{riya}, $\rho(I) \leq r_{J}(I)-1+(e_{2}(I)-1)e_{2}(I)-e_{3}(I)$ raises an intriguing question that  if $R$ is not Cohen-Macaulay but still has dimension two, what possible bounds can be established? In this section, we explore this question in the context of Buchsbaum rings. Note that in  Theorem  \ref{riyi}, for two-dimensional Buchsbaum rings with positive depth, the same arguments follow from Lemma \ref{papa}, which gives  $\rho(I) \leq  r_{J}(I)-1-e_{3}(I)+\widetilde{e}_{3}(I).$ But the bound on $\widetilde{e}_{3}(I)$ in Lemma \ref{mainak} does not hold for the non Cohen-Macaulay rings. Therefore, we need a change of rings to understand whether a similar bound on $\widetilde{e}_{3}(I)$ exists or not.

For a Buchsbaum ring $R,$ we have a finite integral extension into a  ring $\mathbb S$, called the $S_{2}$-fication of $R,$ We briefly recall the $S_{2}$-fication of  generalized Cohen-Macaulay rings from the literature, for further reading and for the elementary properties of the $S_{2}$-fication, we redirect the readers to \cite[Theorem 3.2]{aoyama} and \cite[Proposition 4.1, Theorem 4.2]{gghv}.

\textbf{$S_{2}$-fication of  generalized Cohen-Macaulay rings:}\cite[Section 4]{gghv} If $R$ is a Noetherian ring with the total ring of fractions $\mathbb K$, the smallest finite extension $\phi: R\longrightarrow \mathbb S \subseteq \mathbb K$ with the $S_{2}$ property is called the $S_{2}$-fication of $R$. The map $\phi$ or $\mathbb S$ is referred as the $S_{2}$-fication of $R.$ 
Let $(R, \mathfrak m)$ be a two-dimensional Buchsbaum ring with  $\depth R >0$  and $\mathbb S$ be the $S_{2}$-fication of $R$. We may assume $R$ to be complete (see \cite[Chapter 1, Lemma 1.13]{sv}). From \cite[Proposition 4.1(2)]{gghv}, we get that $\mathbb S$ is a semi-local ring, therefore, let $\{\mathfrak M_{1},\ldots, \mathfrak M_{l}\}$ be the maximal ideals of $\mathbb S,$ and let $\mathbb S_{1},\ldots,\mathbb S_{l}$ denote the corresponding localizations. For each $i=1,2,\ldots,l,$ let $f_{i}$ be the relative degree $[\mathbb S/\mathfrak M_{i}:R/\mathfrak m].$ If $W$ is an $R$-module of finite length which is also an $\mathbb S$-module, then from the proof of \cite[Theorem 4.3]{gghv}, we have that
\begin{equation} \label{vivek}
    \lambda_{R}(W)= \displaystyle \sum_{i=1}^{l}\lambda_{\mathbb S_{i}}(W\mathbb S_{i})f_{i}.
\end{equation}

The following lemma gives the relationship between the Hilbert coefficients of the ideal $I\mathbb S$ and the ideals $I \mathbb S_{i}$ for $i=1,2,\ldots,l.$ While the proof of this lemma is analogous to the arguments in  \cite[Theorem 4.3]{gghv},   we include the steps for completeness.

\begin{lemma} \label{bintere}
     Let $(R, \mathfrak m)$ be a two-dimensional Buchsbaum ring with  $\depth R >0$  and $\mathbb S$ be the $S_{2}$-fication of $R$. Let $I$ be an $\mathfrak m$-primary ideal. Then 
     \begin{enumerate}[(i)]
         \item  $e^{R}(I\mathbb S)=\displaystyle \sum_{i=1}^{l}e^{\mathbb S_{i}}(I\mathbb S_{i})f_{i}$
         \item    $e_{1}^{R}(I\mathbb S)=\displaystyle \sum_{i=1}^{l}e_{1}^{\mathbb S_{i}}(I\mathbb S_{i})f_{i}$
         \item   $e_{2}^{R}(I\mathbb S)=\displaystyle \sum_{i=1}^{l}e_{2}^{\mathbb S_{i}}(I\mathbb S_{i})f_{i}$ 
         \item $e_{3}^{R}(I\mathbb S)=\displaystyle \sum_{i=1}^{l}e_{3}^{\mathbb S_{i}}(I\mathbb S_{i})f_{i}.$
     \end{enumerate}
\end{lemma}

\begin{proof}
  Applying formula  (\ref{vivek}) to the modules $I^{n} \mathbb S/I^{n+1} \mathbb S$ and $\mathbb S/I^{n+1} \mathbb S$, we get $\lambda_{R}\left(I^{n} \mathbb S/I^{n+1} \mathbb S\right) = \displaystyle \sum_{i=1}^{l}  \lambda_{\mathbb S_{i}}\left(I^{n} \mathbb S_{i}/I^{n+1} \mathbb S_{i}\right)f_{i}$ and $ \lambda_{R}\left( \mathbb S/I^{n+1} \mathbb S\right) = \displaystyle \sum_{i=1}^{l}  \lambda_{\mathbb S_{i}}\left( \mathbb S_{i}/I^{n+1} \mathbb S_{i}\right)f_{i}$. Then for sufficiently large $n$, we have the following equations respectively:
\begin{equation} \label{seven}
    e^{R}(I\mathbb S)(n+1)-e_{1}^{R}(I\mathbb S)= \displaystyle \sum_{i=1}^{l}e^{\mathbb S_{i}}(I\mathbb S_{i})f_{i}(n+1)-\displaystyle \sum_{i=1}^{l}e_{1}^{\mathbb S_{i}}(I\mathbb S_{i})f_{i},
\end{equation}
\begin{equation} \label{eight}
      \displaystyle \sum_{j=0}^{2}(-1)^{j} e_{j}^{R}(I\mathbb S) \binom{n+2-j}{2-j}= 
  \displaystyle \sum_{j=0}^{2}(-1)^{j}\left(  \displaystyle \sum_{i=1}^{l}e_{j}^{\mathbb S_{i}}(I\mathbb S_{i})f_{i}\right)\binom{n+2-j}{2-j}.
\end{equation}

On comparing the coefficient of $n$ in equation (\ref{seven}), we get $e^{R}(I\mathbb S)=\displaystyle \sum_{i=1}^{l}e^{\mathbb S_{i}}(I\mathbb S_{i})f_{i}$ and
         $e_{1}^{R}(I\mathbb S)=\displaystyle \sum_{i=1}^{l}e_{1}^{\mathbb S_{i}}(I\mathbb S_{i})f_{i}.$ 
For $n=0$ in equation (\ref{eight}), we get $e^{R}(I\mathbb S)-e_{1}^{R}(I\mathbb S)+e_{2}^{R}(I\mathbb S)= \displaystyle \sum_{i=1}^{l}e^{\mathbb S_{i}}(I\mathbb S_{i})f_{i}-\displaystyle \sum_{i=1}^{l}e_{1}^{\mathbb S_{i}}(I\mathbb S_{i})f_{i}+\displaystyle \sum_{i=1}^{l}e_{2}^{\mathbb S_{i}}(I\mathbb S_{i})f_{i}$. Therefore, we have  $e_{2}^{R}(I\mathbb S)=\displaystyle \sum_{i=1}^{l}e_{2}^{\mathbb S_{i}}(I\mathbb S_{i})f_{i}.$

Now consider $\displaystyle \sum_{j=0}^{n}\lambda_{R}\left( \mathbb S/I^{j+1} \mathbb S\right) =\displaystyle \sum_{j=0}^{n} \displaystyle \sum_{i=1}^{l}  \lambda_{\mathbb S_{i}}\left(\mathbb S_{i}/I^{j+1} \mathbb S_{i}\right)f_{i},$ then for sufficiently large $n$, we get that
\begin{equation} \label{nine1}
   \displaystyle \sum_{j=0}^{3}(-1)^{j} e_{j}^{R}(I\mathbb S) \binom{n+3-j}{3-j}= 
  \displaystyle \sum_{j=0}^{3}(-1)^{j}\left(  \displaystyle \sum_{i=1}^{l}e_{j}^{\mathbb S_{i}}(I\mathbb S_{i})f_{i}\right)\binom{n+3-j}{3-j}.
\end{equation}

For $n=0$ in equation (\ref{nine1}) and substituting values of $e^{R}(I\mathbb S)$, $e_{1}^{R}(I\mathbb S)$, and $e_{2}^{R}(I\mathbb S)$ from above, we get $e_{3}^{R}(I\mathbb S)=\displaystyle \sum_{i=1}^{l}e_{3}^{\mathbb S_{i}}(I\mathbb S_{i})f_{i}.$
\end{proof}

\begin{remark}
   \begin{enumerate} [(i)]
    \label{remark1}\item  Let $(R, \mathfrak m)$ be a two-dimensional Buchsbaum ring with  $\depth R >0$  and $\mathbb S$ is the $S_{2}$-fication of $R$ and  $I$ an $\mathfrak m$-primary ideal with minimal reduction $J$. Then
      from the proof of \cite[Theorem 4.3]{gghv}, we have 
 $r_{J\mathbb S}(I\mathbb S)=\max\{r_{J\mathbb S_{1}}(I\mathbb S_{1}),\ldots,r_{J\mathbb S_{l}}(I\mathbb S_{l})\}.$ 

 \item Note that $\rho(I\mathbb S)=\max\{\rho(I\mathbb S_{i}),\ldots,\rho(I\mathbb S_{l})\}.$ To see this, consider the localization maps $\phi_{i}:\mathbb S \longrightarrow \mathbb S_{i},$ for every $i=1,2,\ldots, l.$ Then clearly $\rho(I\mathbb S_{i}) \leq \rho(I \mathbb S)$ for all $i=1,2,\ldots,l,$ this implies that  $\displaystyle \max_{1\leq i\leq l}(\rho(I \mathbb S_{i})) \leq \rho(I \mathbb S).$ Now let $m=\max\{\rho(I\mathbb S_{i}),\ldots,\rho(I\mathbb S_{l})\},$ and let $x\mathbb S_{i}$ be the image of $x$ is $\mathbb S_{i}$ for every $i=1,2,\ldots,l,$ where $x \in I\backslash I^{2}$ is a superficial element for $I.$
 Then for every $n \geq m$, $(I^{n+1}\mathbb S_{i}:x\mathbb S_{i})=I^{n}\mathbb S_{i}$ for all $i=1,2,\ldots,l$ and for every maximal ideal $\mathfrak M_{i}.$ Therefore, from the local properties of localization, we have $(I^{n+1}\mathbb S:x\mathbb S)=I^{n}\mathbb S$ for every $n \geq m,$ this implies $\rho (I \mathbb S)\leq m.$ Thus,   $\rho(I\mathbb S)=\max\{\rho(I\mathbb S_{i}),\ldots,\rho(I\mathbb S_{l})\}$.
   \end{enumerate}
\end{remark}

In the following lemma, we establish the relationships between $e_{i}(I)$ and $e_{i}^{R}(I \mathbb S)$ for $0\leq i\leq 3$ under the assumption that $I=I\mathbb S.$

 \begin{lemma} \label{coefflemma}
    Let $(R, \mathfrak m)$ be a two dimensional Buchsbaum ring with $\depth R>0$ and $\mathbb S$ be the $S_{2}$-fication of $R.$ Let $I$ be an $\mathfrak m$-primary ideal with minimal reduction $J$ and assume that $I=I\mathbb S$. Then we have the following:
    \begin{enumerate}[(i)]
        \item \label{eo} $e_{i}(I)=e_{i}^{R}(I \mathbb S)$ for $0\leq i \leq 1,$
        \item \label{e1} $e_{2}(I)=e_{2}^{R}(I \mathbb S)+e_{1}(J),$
        \item $e_{3}(I)=e_{3}^{R}(I \mathbb S).$
    \end{enumerate}
\end{lemma}
\begin{proof}
    From the proof of \cite[Proposition 4.1]{gghv}, consider the short exact sequence
    \begin{equation} \label{gotoseq}
        0\longrightarrow R\longrightarrow \mathbb S\longrightarrow D\longrightarrow 0.
     \end{equation}
     Note that $D \cong \mathcal{H}^{1}_{\mathfrak m}(R)$ is a module of finite length. Tensoring (\ref{gotoseq}) with $R/I^{n+1},$ we get
     \begin{equation*}
         0 \longrightarrow \frac{I^{n+1}\mathbb S \cap R}{I^{n+1}} \longrightarrow R/I^{n+1}\longrightarrow \mathbb S/I^{n+1} \mathbb S\longrightarrow D/I^{n+1}D\longrightarrow 0.
     \end{equation*}
     Since $I \mathbb S=I$, we have $I^{n}\mathbb S=\underbrace{(I \mathbb S)\cdots(I \mathbb S)}_{n\text{-times}}=\underbrace{I \cdots I }_{n\text{-times}}=I^n$ for all $n.$ Therefore, $\dfrac{I^{n+1}\mathbb S \cap R}{I^{n+1}}=0$ for all $n.$ Also, from \cite[Chapter 2, Proposition 2.1 $(iii)$]{sv},  $\mathfrak m \mathcal{H}^{1}_{\mathfrak m}(R)=0,$ thus, $I^{n+1}\mathcal{H}^{1}_{\mathfrak m}(R)=0$ for all $n,$ which implies $I^{n+1}D=0$ for all $n.$  Therefore, we get the following short exact sequence:
      \begin{equation} \label{modifiedgoto}
         0 \longrightarrow  R/I^{n+1}\longrightarrow \mathbb S/I^{n+1} \mathbb S\longrightarrow D\longrightarrow 0.
     \end{equation}
     From the exact sequence (\ref{modifiedgoto}), we have $\lambda(R/I^{n+1})+\lambda_{R}(D)=\lambda_{R}(\mathbb S/I^{n+1} \mathbb S).$ Then for sufficiently large $n,$ we have the following equation:
     \begin{equation} \label{poly1}
         e_{0}(I)\binom{n+2}{2}-e_{1}(I)(n+1)+e_{2}(I)+\lambda_{R}(D)=e_{0}^{R}(I \mathbb S)\binom{n+2}{2}-e_{1}^{R}(I \mathbb S)(n+1)
         +e_{2}^{R}(I \mathbb S).
     \end{equation}
     On comparing the coefficients in (\ref{poly1}), we get  $e_{0}(I)=e_{0}^{R}(I \mathbb S),$ $e_{1}(I)=e_{1}^{R}(I \mathbb S)$ and
         $e_{2}(I)+\lambda_{R}(D)=e_{2}^{R}(I \mathbb S).$ Furthermore, $\lambda_{R}(D)=\lambda\left(\mathcal{H}^{1}_{\mathfrak m}(R)\right)=-e_{1}(J)$ because $R$ is Buchsbaum \cite[
         Chapter 1, Propositions 2.6 and 2.7]{sv}. Thus, $e_{2}(I)=e_{2}^{R}(I \mathbb S)+e_{1}(J).$

         Now consider $\displaystyle \sum_{j=0}^{n}\left[\lambda(R/I^{j+1})+\lambda_{R}(D)\right]=\displaystyle \sum_{j=0}^{n}\lambda_{R}\left( \mathbb S/I^{j+1} \mathbb S\right),$ then for sufficiently large $n$, we get that
\begin{equation} \label{nine}
  \displaystyle \sum_{j=0}^{3}(-1)^{j} e_{j}(I) \binom{n+3-j}{3-j} +(n+1)\lambda_{R}(D)=\displaystyle \sum_{j=0}^{3}(-1)^{j} e_{j}^{R}(I\mathbb S) \binom{n+3-j}{3-j}.
\end{equation}
Substituting values of $e_{0}(I)$, $e_{1}(I)$ and $e_{2}(I)$ from $(\ref{eo})$ and $(\ref{e1})$, we get $e_{3}(I)=e_{3}^{R}(I \mathbb S).$
\end{proof}

In the following theorem,  we give a bound on $\rho(I)$ for two-dimensional Buchsbaum rings.

\begin{theorem} \label{rudra}
    Let $(R, \mathfrak m)$ be a two dimensional Buchsbaum ring with  $\depth R >0$  and $\mathbb S$ be the $S_{2}$-fication of $R$. Let $I$ be an $\mathfrak m$-primary ideal with minimal reduction $J$ and assume that $I=I\mathbb S$. If $r_{J}(I) < \rho(I),$ then $\rho(I) \leq r_{J}(I)-1+(e_{2}(I)-e_{1}(J)-1)(e_{2}(I)-e_{1}(J))-e_{3}(I)$.
\end{theorem}
\begin{proof}
    We first note that for any $n \geq \rho(I \mathbb S),$ we have $\widetilde{I^{n}} \subseteq \widetilde{I^{n}}\mathbb S \cap R \subseteq \widetilde{I^{n}\mathbb S}\cap R= I^{n}\mathbb S \cap R=I^{n}.$ This implies $\widetilde{I^{n}}=I^{n}$ for all $n \geq \rho(I \mathbb S),$ and therefore, $\rho(I) \leq \rho(I \mathbb S).$ Without the loss of generality, let $\rho(I \mathbb S)=\rho(I \mathbb S_{1}).$ It is not hard to see that $r_{J\mathbb S_{1}}(I \mathbb S_{1}) \leq r_{J\mathbb S}(I\mathbb S)$ and $r_{J\mathbb S}(I\mathbb S) \le r_{J}(I)$.  Thus, we get $r_{J \mathbb S_{1}}(I \mathbb S_{1}) < \rho(I \mathbb S_{1}).$ Therefore, from Proposition \ref{riya}, we have 
    \begin{align} \nonumber
        \rho(I \mathbb S_{1}) & \leq  r_{J \mathbb S_{1}}(I \mathbb S_{1})-1+ (e_{2}^{\mathbb S_{1}}(I\mathbb S_{1})-1)e_{2}^{\mathbb S_{1}}(I\mathbb S_{1})-e_{3}^{\mathbb S_{1}}(I\mathbb S_{1})\\ \label{anji}
&\leq r_{J \mathbb S_{1}}(I \mathbb S_{1})-1+ \displaystyle \sum_{i=1}^{l}\left[(e_{2}^{\mathbb S_{i}}(I\mathbb S_{i})-1)e_{2}^{\mathbb S_{i}}(I\mathbb S_{i})-e_{3}^{\mathbb S_{i}}(I\mathbb S_{i})\right]\\   \nonumber
&\leq r_{J \mathbb S_{1}}(I \mathbb S_{1})-1+ \displaystyle \sum_{i=1}^{l}\left[\left((e_{2}^{\mathbb S_{i}}(I\mathbb S_{i})-1)e_{2}^{\mathbb S_{i}}(I\mathbb S_{i})-e_{3}^{\mathbb S_{i}}(I\mathbb S_{i})\right)f_{i}\right]\\ \nonumber
& = r_{J \mathbb S_{1}}(I \mathbb S_{1})-1+ \displaystyle \sum_{i=1}^{l}\left[\left(e_{2}^{\mathbb S_{i}}(I\mathbb S_{i})\right)^{2}f_{i}-e_{2}^{\mathbb S_{i}}(I\mathbb S_{i})f_{i}-e_{3}^{\mathbb S_{i}}(I\mathbb S_{i})f_{i}\right]\\ \nonumber
& \leq r_{J \mathbb S_{1}}(I \mathbb S_{1})-1+ \displaystyle \sum_{i=1}^{l}\left[(e_{2}^{\mathbb S_{i}}(I\mathbb S_{i})f_{i})^{2}-e_{2}^{\mathbb S_{i}}(I\mathbb S_{i})f_{i}-e_{3}^{\mathbb S_{i}}(I\mathbb S_{i})f_{i}\right]\\ \label{genda}
           &\leq r_{J \mathbb S_{1}}(I \mathbb S_{1})-1+ \displaystyle \left[\sum_{i=1}^{l}(e_{2}^{\mathbb S_{i}}(I\mathbb S_{i})f_{i})\right]^{2}-\sum_{i=1}^{l}\left[e_{2}^{\mathbb S_{i}}(I\mathbb S_{i})f_{i}+e_{3}^{\mathbb S_{i}}(I\mathbb S_{i})f_{i}\right]\\   \nonumber
        & \leq r_{J \mathbb S}(I \mathbb S)-1+ (e_{2}^{R}(I \mathbb S)-1)e_{2}^{R}(I \mathbb S)-e_{3}^{R}(I \mathbb S) \quad \text{ (from Lemma \ref{bintere} and Remark \ref{remark1} (i))}\\ \nonumber
        & \leq r_{J}(I)-1+(e_{2}(I)-e_{1}(J)-1)(e_{2}(I)-e_{1}(J))-e_{3}(I)  \quad  \text{ (from Lemma \ref{coefflemma}}). 
      \end{align}

        Note that the inequality in (\ref{anji})  holds by the proof of Theorem \ref{riyi} and Lemma \ref{mainak},  $e_{3}^{\mathbb S_{i}}(I\mathbb S_{i}) \leq \widetilde{e}_{3}^{\mathbb S_{i}}(I\mathbb S_{i}) \leq e_{2}^{\mathbb S_{i}}(I\mathbb S_{i})(e_{2}^{\mathbb S_{i}}(I\mathbb S_{i})-1)$ for all $1 \leq i \leq l.$ Further, since $\mathbb S_{i}$ is Cohen- Macaulay of dimension 2 for all $1 \leq i \leq l,$ thus the inequality in (\ref{genda}) holds since $e_{2}^{\mathbb S_{i}}(I\mathbb S_{i})\geq 0$ for all $1 \leq i \leq l$ from \cite[Proposition 3.1]{rv}.
\end{proof}

As an immediate consequence, we have the following corollary.
\begin{corollary}
     Let $(R, \mathfrak m)$ be a two dimensional Buchsbaum ring with  positive depth. For any minimal reduction $J$ of $\mathfrak m$, if $r_{J}(\mathfrak m) < \rho(\mathfrak m),$ then $$\rho(\mathfrak m) \leq r_{J}(\mathfrak m)-1+(e_{2}(\mathfrak m)-e_{1}(J)-1)(e_{2}(\mathfrak m)-e_{1}(J))-e_{3}(\mathfrak m).$$
\end{corollary}

\begin{proof}
Using \cite[Proposition 4.1]{gghv}  and  from \cite[Chapter 2, Proposition 2.1 $(iii)$]{sv}, $\mathfrak{m} \mathcal{H}^{1}_{\mathfrak{m}}(R) = 0$, we have that  $\mathfrak m \mathbb S=\mathfrak m.$ Therefore, the conclusion follows directly from Theorem \ref{rudra}.
\end{proof}

The following example illustrates Theorem \ref{rudra}.

\begin{example} \label{example 4.6}
    \cite[Example 4.7]{gghv} Let $R=\mathbb Q[[X,Y,Z,W]]/(X,Y)\cap(Z,W),$ where $X,Y,Z,W$ are variables. Then $R$ is a two-dimensional Buchsbaum ring with positive depth. Let $x,y,z,w$ denote the images of $X,Y,Z,W$ in $R,$ and let $I=(x^7,x^{6}y,x^{2}y^{5},y^{7},z,w).$ Note that the $S_{2}$-fication of $R$ is $\mathbb{S}=\mathbb Q[[z,w]]\bigoplus \mathbb Q[[x,y]],$ so that $I=I\mathbb S.$ By CoCoA, the Hilbert series is
    \begin{equation*}
        h_{I}(t)=\frac{35+8t+6t^{2}+2t^{3}-t^{5}}{(1-t)^2}.
    \end{equation*}

Hence, $e(I)=50, e_{1}(I)=21$, $e_{2}(I)=2$ and $e_{3}(I)=-8.$ Using Macaulay 2, $J=(\frac{2}{3}x^{7}+\frac{9}{8}x^{6}y+10x^{2}y^{5}+\frac{8}{7}y^{7}+\frac{6}{5}z+w,\frac{3}{2}x^{7}+\frac{2}{3}x^{6}y+4x^{2}y^{5}+4y^{7}+\frac{8}{3}z+10w),$ is a minimal reduction of $I$ with $r_{J}(I)=3$ and $e_{1}(J)=-1.$ Note that $x^{17}y^{4} \in (I^4:I)$ but $x^{17}y^{4} \notin I^3,$ therefore, $\widetilde{I^{3}} \neq I^{3}.$ 
Thus, from Theorem \ref{rudra}, we have $\rho(I) \leq r_{J}(I)-1+(e_{2}(I)-e_{1}(J)-1)(e_{2}(I)-e_{1}(J))-e_{3}(I)=16.$
\end{example}

 We end this section with the following question.

 \begin{question}
     Is $\rho(I) \leq r_{J}(I)-1+(e_{2}(I)-e_{1}(J)-1)(e_{2}(I)-e_{1}(J))-e_{3}(I)$ for any $\mathfrak m$-primary ideal $I$ with minimal reduction $J$ in a two-dimensional Buchsbaum ring with positive depth under the condition that $r_{J}(I)<\rho(I)?$ 
 \end{question}

\section{Bounds on the Castelnuovo-Mumford Regularity}
In this section, we give applications of the theorems of the last sections by giving bounds on the Castelnuovo-Mumford regularity. In \cite{rtt}, Rossi, Trung and Trung established a relationship between the regularity of the associated graded ring of $I$ and the seemingly unrelated Ratliff-Rush closure of an ideal. They characterized $\reg \mathcal{R}(I)$ in terms of $r_{J}(I)$ and $\rho(I)$ for an $\mathfrak m$-primary ideal $I$ in a two-dimensional Buchsbaum ring \cite[Theorem 2.4]{rtt}.  Note that Ooishi \cite[Lemma 4.8]{ooishi2} showed that there is always an equality 
$\reg G(I) = \mathrm{reg}\, \mathcal{R}(I)$; 
see also \cite[Corollary 3.3]{trung}.
In the following proposition, drawing inspirations from their result, we characterize $\reg G(I)$ in terms of  $\rho(I),\ldots,\rho\left(I/(x_{1},\ldots,x_{d-2})\right),$ and $r_{J}(I)$, where $x_{1},\ldots,x_{d-2} \in I$ is a superficial sequence for $I.$

 \begin{theorem} \label{regularity}
          Let $(R, \mathfrak m)$ be a Cohen-Macaulay local ring of dimension $d \geq 2,$ and $I$ an $\mathfrak m$-primary ideal with minimal reduction $J=(x_{1}, \ldots, x_{d}),$ where $x_{1}, \ldots, x_{d}  $ is a superficial sequence for $I,$ then
          \begin{center}
               $\reg G(I)=\max\{\rho(I), \rho(I/(x_{1})),\ldots,\rho(I/(x_{1},x_{2},\ldots,x_{d-2})) ,r_{J}(I)\}.$
          \end{center}
          \end{theorem}

\begin{proof}
    We prove by induction on $d.$ For $d=2,$ the conclusion follows from \cite[Theorem 2.4]{rtt}, as $\reg G(I)= \max\{r_{J}(I),\rho(I)\}.$ We assume the proposition for $d-1$ and prove for $d.$ Set $\bar{R}=R/(x_{1})$ and $\bar{I}=I/(x_{1})$, then $\bar{R}$ is a $d-1$ dimensional Cohen-Macaulay local ring. By the induction hypothesis, we have 
    $\reg G(\bar{I})=\max\{r_{\bar{J}}(\bar{I}), \rho(\bar{I}), \rho\left(\bar{I}/(x_{2})\right),\ldots, \rho\left(\bar{I}/(x_{2},\ldots,x_{d-2})\right)\}=\max\{r_{\bar{J}}(\bar{I}), \rho(I/(x_{1})), \rho\left(I/(x_{1},x_{2})\right),\ldots, \rho\left(I/(x_{1},\ldots,x_{d-2})\right)\}$. This implies 
    \begin{equation*}
        a_{i}(G(\bar{I}))+i \leq \max\{r_{\bar{J}}(\bar{I}), \rho(I/(x_{1})),\ldots, \rho\left(I/(x_{1},\ldots,x_{d-2})\right)\} \text{ for all } i=1,\ldots,d-1.
    \end{equation*}

    We claim that $a_{i+1}(G(I))+(i+1) \leq a_{i}(G(\bar{I}))+i$ for all $i=1,\ldots, d-1.$ 
    Note that 
    \begin{equation*}
        \frac{\bar{I}^{n}}{\bar{I}^{n+1}}=\frac{I^{n}+(x_{1})}{I^{n+1}+(x_{1})} \cong \frac{I^{n}}{I^{n}\cap(I^{n+1}+(x_{1}))}=\frac{I^{n}}{((x_{1})\cap I^{n})+I^{n+1}}.
    \end{equation*}
So letting $W=\displaystyle \bigoplus_{n=0}^{\infty}\frac{x_{1}(I^{n}:x_{1})+I^{n+1}}{x_{1}I^{n-1}+I^{n+1}}$, there is an exact sequence 
\begin{equation} \label{les}
    0\longrightarrow W \longrightarrow \frac{G(I)}{x_{1}^{*}G(I)}=\displaystyle \bigoplus_{n=0}^{\infty}\frac{I^{n}}{x_{1}I^{n-1}+I^{n+1}}\longrightarrow G(\bar{I})\longrightarrow 0.
\end{equation}
Since $x_{1}$ is superficial for $I$, we get $W_{n}=0$ for sufficiently large $n.$ As $I$ is $\mathfrak m$-primary, we deduce that $\lambda(W) < \infty.$ Moreover, $x_{1}^{*}=x_{1}+I^{2}$ is $G(I)$-filter-regular. Taking the long exact sequence of cohomology associated to (\ref{les}), we deduce 
\begin{equation*}
    \mathcal{H}^{i}(G(\bar{I})) \cong \mathcal{H}^{i}(G(I)/x_{1}^{*}G(I)) \text{ for all }  i\geq 1.
\end{equation*}
Invoking the proof of \cite[Lemma 2.3]{trung2}, and the fact that $x_{1}^{*}$ is $G(I)$-filter-regular, we deduce
\begin{equation*}
      a_{i+1}(G(I)) \leq a_{i}(G(I)/x_{1}^{*}G(I)) -1=a_{i}(G(\bar{I}))-1 \text{  for all i  } \geq 1, 
\end{equation*}
as claimed.
    Therefore, for all $i=1,\ldots,d-1,$ we get
        \begin{align}   \label{10} \nonumber
            a_{i+1}(G(I))+(i+1) & \leq a_{i}(G(\bar{I}))+i\\ \nonumber
            & \leq \max\{r_{\bar{J}}(\bar{I}), \rho(I/(x_{1})), \rho\left(I/(x_{1},x_{2})\right),\ldots, \rho\left(I/(x_{1},x_{2},\ldots,x_{d-2})\right)\}\\
            & \leq \max\{r_{J}(I), \rho(I/(x_{1})), \rho\left(I/(x_{1},x_{2})\right),\ldots, \rho\left(I/(x_{1},x_{2},\ldots,x_{d-2})\right)\}.
        \end{align}

From \cite[Proposition 4.7(2)]{tjp}, we have $\mathcal{H}^{0}\left(L^{\bar{I}}(\bar{R})\right)=\displaystyle \bigoplus_{n\geq 0}\widetilde{\overline{I^{n+1}}}/\overline{I^{n+1
}}$. Thus, $\mathcal{H}^{0}\left(L^{\bar{I}}(\bar{R})\right)_{n}=0$ for all $n \geq \rho\left(I/(x_{1})\right)-1,$ and from \cite[Theorem 6.4(1)]{tjp}, we have $\mathcal{H}^{1}\left(L^{I}(R)\right)_{n}=0$ for sufficiently large $n.$ Therefore, from the long exact sequence (\ref{4}) of local cohomology modules:
\begin{equation*} 
    \begin{aligned}
        0 \longrightarrow \mathcal{B}(x,R)  &  \longrightarrow \mathcal{H}^{0}\left(L^{I}(R)\right)(-1) \longrightarrow \mathcal{H}^{0}\left(L^{I}(R)\right) \longrightarrow \mathcal{H}^{0}\left(L^{\bar{I}}(\bar{R})\right) \\
     & \longrightarrow \mathcal{H}^{1}\left(L^{I}(R)\right)(-1) \longrightarrow \mathcal{H}^{1}\left(L^{I}(R)\right) \longrightarrow \mathcal{H}^{1}\left(L^{\bar{I}}(\bar{R})\right) \ldots
    \end{aligned}
\end{equation*} the map  $f_{n}:\mathcal{H}^{1}\left(L^{I}(R)\right)_{n-1}\longrightarrow \mathcal{H}^{1}\left(L^{I}(R)\right)_{n}$ is injective for all $n \geq \rho\left(I/(x_{1})\right)-2.$ Thus, $\mathcal{H}^{1}\left(L^{I}(R)\right)_{n}=0$ for all $n \geq \rho\left(I/(x_{1})\right)-2.$ From \cite[Proposition 4.7(2)]{tjp}, we have $\mathcal{H}^{0}\left(L^{I}(R)\right)=\displaystyle \bigoplus_{n\geq 0}\widetilde{I^{n+1}}/I^{n+1
}$, thus $\mathcal{H}^{0}\left(L^{I}(R)\right)_{n}=0$ for all $n \geq \rho(I)-1.$ Therefore, from the long exact sequence (\ref{csk}),  we have $\mathcal{H}^{1}\left(G(I)\right)_{n}=0$ for all $n \geq \max\{ \rho(I), \rho\left(I/(x_{1})\right)-2\},$ This implies
    \begin{equation} \label{11}
        a_{1}(G(I))+1 \leq \max\{ \rho(I), \rho\left(I/(x_{1})\right)\}.
    \end{equation}

From inequalities (\ref{10}) and (\ref{11}), we get
\begin{equation*}
    a_{i}(G(I))+i \leq  \max\{r_{J}(I), \rho(I), \rho(I/(x_{1})),\ldots, \rho\left(I/(x_{1},\ldots,x_{d-2})\right)\} \text{ for all } i=1,\ldots,d.
\end{equation*}
Therefore,\begin{equation*}
       \reg G(I) \leq  \max\{r_{J}(I), \rho(I),\ldots, \rho\left(I/(x_{1}\ldots,x_{d-2})\right)\}.
   \end{equation*}
Set $R_{i}=R/(x_{1},\ldots,x_{i})$ for $i=0,\ldots,d-2,$ so that $R_{0}=R.$ From \cite[Lemma 1.2 and Theorem 1.3]{rtt}, we have $r_{J}(I) \leq \reg G(I)$, and from \cite[Proposition 2.1(ii)]{rtt}, $\rho(I/(x_{1},\ldots,x_{i})) \leq \reg G(I/(x_{1},\ldots,x_{i}))$  for all $i=0,\ldots,d-2.$ 
Using the exact sequence (\ref{les}) from above, the fact that  $\lambda(W) < \infty,$ $x_{1}^{*}$ is $G(I)$-filter-regular, we deduce that 
\begin{equation*}
    \reg G(\bar{I})=\reg G(I/(x_{1}))\leq \reg \frac{G(I)}{x_{1}^{*}G(I)} \leq \reg G(I).
\end{equation*}
By induction, we get $\reg G(I/(x_{1},\ldots,x_{d-2})) \leq \cdots \leq \reg G(I),$ hence $\reg G(I) \geq\rho(I/(x_{1},\ldots,x_{i})) $ for all $i=0,\ldots,d-2.$
Thus,
\begin{equation*}
    \reg G(I) \geq  \max\{r_{J}(I), \rho(I),\ldots, \rho\left(I/(x_{1},\ldots,x_{d-2})\right)\} 
\end{equation*}
Therefore, we get $\reg G(I)=\max\{\rho(I), \rho(I/(x_{1})),\ldots,\rho(I/(x_{1},x_{2},\ldots,x_{d-2})) ,r_{J}(I)\}.$
\end{proof}

The formula for $\reg G(I)$ in Theorem \ref{regularity}, provides an effective tool for the computation of $\reg G(I)$ especially under the assumption that the Ratliff-Rush filtration of $I$ behaves well mod a superficial sequence. As an immediate consequence, we have the following corollary.
 \begin{corollary} \label{cor5.2}
          Let $(R, \mathfrak m)$ be a Cohen-Macaulay local ring of dimension $d \geq 2,$ and $I$ an $\mathfrak m$-primary ideal with minimal reduction $J=(x_{1}, \ldots, x_{d}),$ where $x_{1}, \ldots, x_{d}  $ is a superficial sequence for $I.$ Suppose the Ratliff-Rush filtration with respect to $I$ behaves well mod a superficial sequence $x_{1},  \ldots, x_{d-2},$ then
          \begin{center}
     $\reg G(I)=\max\{ r_{J}(I),\rho(I)\}.$
    \end{center}

    Further, if $r_{J}(I) < \rho(I),$ then $\reg G(I) \leq r_{J}(I)-1+(-1)^{d+1}(e_{d+1}(I)-\widetilde{e}_{d+1}(I)).$
    \end{corollary}

    \begin{proof}
        Since the Ratliff-Rush filtration behaves well mod a superficial sequence $x_{1},\ldots,x_{d-2} $ for $I$, thus, from \cite[Lemma 4.9]{sp}, we have    $\rho(I) \geq \rho\left(I/(x_{1})\right) \geq \ldots \geq \rho\left(I/(x_{1},  \ldots, x_{d-2})\right).$ Therefore, from Theorem \ref{regularity}, we get that  $\reg G(I)=\max\{ r_{J}(I),\rho(I)\}.$ 
 Next, let $r_{J}(I) < \rho(I),$ then from Theorem \ref{riyi}, we have $\reg G(I) \leq r_{J}(I)-1+(-1)^{d+1}(e_{d+1}(I)-\widetilde{e}_{d+1}(I)).$
    \end{proof}

Miranda-Neto and Queiroz proposed \cite[Question 4.4]{cmn} that in a Cohen-Macaulay ring of dimension $d\geq 2,$ for an $\mathfrak m$-primary ideal $I$ with minimal reduction $J \subseteq I$, if $\widetilde{I^{r_{J}(I)}}= I^{r_{J}(I)}$, then is $\reg G(I)= r_{J}(I)?$  The following corollary gives partial answer to this question  under the hypothesis that the Ratliff-Rush filtration with respect to $I$ behaves well mod a superficial sequence.

\begin{corollary} \label{nneed}
     Let $(R, \mathfrak m)$ be a Cohen-Macaulay local ring of dimension $d \geq 2,$ $I$ an $\mathfrak m$-primary ideal and $J=(x_{1}, \ldots, x_{d}),$ where $x_{1}, \ldots, x_{d}  $ is a superficial sequence for $I$. Suppose that the Ratliff-Rush filtration with respect to $I$ behaves well mod a superficial sequence $x_{1},\ldots,x_{d-2} $ and $\widetilde{I^{r_{J}(I)}}= I^{r_{J}(I)}$. Then $\reg G(I)= r_{J}(I).$ 
\end{corollary}

\begin{proof}
 As the Ratliff-Rush filtration behaves well mod a superficial sequence $x_{1},\ldots,x_{d-2}$ for $I$, thus, from Corollary \ref{cor5.2}, we have $\reg G(I)=\max\{ r_{J}(I),\rho(I)\}.$
 Since $\widetilde{I^{r_{J}(I)}}= I^{r_{J}(I)}$, therefore, from Proposition \ref{maa}, we have $\rho(I) \leq r_{J}(I),$ which implies $\reg G(I)= r_{J}(I).$
\end{proof}

   The following example illustrates the above corollary.
    \begin{example} \cite[Example 3.1]{rv}
        Let $R=\mathbb Q[[x,y,z]]$ and $I=(x^{2}-y^{2},y^{2}-z^{2},xy,yz,zx).$ By CoCoA, the Hilbert series is
        \begin{equation*}
            h_{I}(t)=\frac{5+6t^{2}-4t^{3}+t^{4}}{(1-t)^{3}}.
        \end{equation*}
        Then $e_{2}(I)=e_{3}(I)=0.$ By \cite[Theorems 4.5 and 6.2]{tjp2}, the Ratliff-Rush filtration with respect to $I$ behaves well mod a superficial element. Note that $J=(x^2-y^2,x^2-z^2,xy+yz+zx)$ is a minimal reduction of $I$ with $r_{J}(I)=2$ such that $\widetilde{I^{2}}=I^{2}.$  Therefore, by Corollary \ref{nneed}, $\reg G(I)=r_{J}(I)=2.$
    \end{example}

We now give upper bounds on the regularity of the associated graded ring in terms of higher Hilbert coefficients.
 
\begin{proposition}\label{lama}  
  Let $(R,\mathfrak m)$ be a two dimensional Cohen-Macaulay local ring, and $I$ an $\mathfrak m$-primary ideal. Let $J \subseteq I$ be a minimal reduction of $I.$ Then 
  \footnotesize{
\begin{equation*}
\reg G(I) \leq
    \begin{cases}
        e_{1}(I)-e(I)+\lambda(R/I)+1 &, \text{if } \rho(I) \leq r_{J}(I)\\
          r_{J}(I)-1-e_{3}(I)+\left(\frac{r_{J}(I)-1}{2}\right)(e_{2}(I)-e_{1}(I)+e(I)-\lambda(R/I)+\lambda(\widetilde{I}/I))& , \text{if } r_{J}(I) < \rho(I). 
    \end{cases}
\end{equation*}}
    \end{proposition}

    \begin{proof}
    From \cite[Theorem 2.4]{rtt}, $\reg G(I) = \max \{r_{J}(I), \rho(I)\}.$ If $\rho(I) \leq r_{J}(I),$ then by \cite[Corollary 1.5]{rossi2}, we get $\reg G(I) \leq e_{1}(I)-e_{0}(I)+\lambda(R/I)+1.$ Further, if $r_{J}(I) < \rho(I),$ then by Proposition \ref{mushroom}, we have $\reg G(I) \leq  r_{J}(I)-1-e_{3}(I)+\left(\frac{r_{J}(I)-1}{2}\right)(e_{2}(I)-e_{1}(I)+e(I)-\lambda(R/I)+\lambda(\widetilde{I}/I)).$
\end{proof}

\begin{proposition}\label{sp}  
  Let $(R,\mathfrak m)$ be a two dimensional Cohen-Macaulay local ring, and $I$ an $\mathfrak m$-primary ideal. For  any minimal reduction $J \subseteq I$, if  $r_{J}(I) < \rho(I)$ then
  \begin{equation*}
      \reg G(I) \leq  r_{J}(I)-1+(e_{2}(I)-1)e_{2}(I)-e_{3}(I).
  \end{equation*}
 \end{proposition}

 \begin{proof}
     Since $r_{J}(I) < \rho(I),$ therefore, from \cite[Theorem 2.4]{rtt}, $\reg G(I) = \rho(I).$ Again from Proposition \ref{riya}, we have $\reg G(I) \leq  r_{J}(I)-1+(e_{2}(I)-1)e_{2}(I)-e_{3}(I).$ 
 \end{proof}

 From \cite[Corollary 1.5]{rossi2} we have, $r_{J}(I) \leq e_{1}(I)-e_{0}(I)+\lambda(R/I)+1.$ Thus, if $r_{J}(I) < \rho(I),$ then $\reg G(I) \leq  e_{1}(I)-e_{0}(I)+\lambda(R/I)+(e_{2}(I)-1)e_{2}(I)-e_{3}(I). $ As an immediate corollary, we show that in dimension two, for certain values of $e_{2}(I)$, we get linear bounds on $\reg G(I)$ in terms of higher Hilbert coefficients. 

\begin{corollary}
    Let $(R,\mathfrak m)$ be a two dimensional Cohen-Macaulay local ring, and $I$ an $\mathfrak m$-primary ideal. For  any minimal reduction $J \subseteq I$, if  $r_{J}(I) < \rho(I)$  the following statements hold. 
    \begin{enumerate}[(i)]
        \item If $e_{2}(I)=0$ or $1$, then $\reg G(I) \leq e_{1}(I)-e_{0}(I)+\lambda(R/I)-e_{3}(I).$
      
        \item If $e_{2}(I)=2$, then $\reg G(I) \leq e_{1}(I)-e_{0}(I)+\lambda(R/I)-e_{3}(I)+2.$
        \end{enumerate}  \end{corollary}

        \begin{proof}
            \begin{enumerate} [(i)]
                \item  If $e_{2}(I)=0$ or $1$, then from Proposition \ref{sp}, we have  $\reg G(I) \leq  r_{J}(I)-1-e_{3}(I).$ Further,   from \cite[Corollary 1.5]{rossi2} we have, $r_{J}(I) \leq e_{1}(I)-e_{0}(I)+\lambda(R/I)+1.$
Therefore, $\reg G(I) \leq e_{1}(I)-e_{0}(I)+\lambda(R/I)-e_{3}(I).$

                \item  If $e_{2}(I)=2$, then  $\reg G(I) \leq  r_{J}(I)-e_{3}(I)+1$ from Proposition \ref{sp}. Further,   from \cite[Corollary 1.5]{rossi2} we have, $r_{J}(I) \leq e_{1}(I)-e_{0}(I)+\lambda(R/I)+1.$
Therefore, $\reg G(I) \leq e_{1}(I)-e_{0}(I)+\lambda(R/I)-e_{3}(I)+2.$ \qedhere
            \end{enumerate}
        \end{proof}

We now consider the following examples to compare the bounds given in Proposition \ref{lama} and Proposition \ref{sp} with  the one given by Rossi, Trung and Valla in  \cite[Corollary 3.4]{rtv}.

\begin{example}
    \begin{enumerate}[(i)]
        \item Let $R=\mathbb{Q}[[x,y]]$ and $I=(x^9,x^7y^2,xy^8,y^9)$. By CoCoA, the Hilbert series is
        \begin{equation*}
            h_{I}(t)=\frac{61+6t+7t^{2}+7t^{3}+2t^{4}-t^{6}-t^{7}}{(1-t)^2}.
        \end{equation*}
        Then $ e(I)=81,$ $ e_{1}(I)=36,$ $e_{2}(I)=4$ and $e_{3}(I)=-40.$ By \cite[Theorem 2.1]{elias},  we have
        $\widetilde{I}=(x^9,x^7y^2,x^5y^4,x^3y^6,xy^8,y^9)$ and $\lambda(\widetilde{I}/I)=12.$ By Macaulay 2, $J=(\frac{7}{6}x^9+\frac{3}{5}x^7y^2+\frac{7}{2}xy^8+\frac{7}{4}y^9, \frac{4}{7}x^9+4x^7y^2+\frac{4}{9}xy^8+\frac{3}{2}y^9)$ is a minimal reduction of $I$ with $r_{J}(I)=4$. Note that $x^{13}y^{23} \in (I^{5}:I) \subseteq \widetilde{I^{4}}$ but $x^{13}y^{23} \notin I^{4}$, therefore,  $\widetilde{I^{4}} \neq I^{4}$. Thus  from 
        Proposition \ref{lama},  we get $\reg G(I) \leq r_{J}(I)-1-e_{3}(I)+\left(\frac{r_{J}(I)-1}{2}\right)(e_{2}(I)-e_{1}(I)+e(I)-\lambda(R/I)+\lambda(\widetilde{I}/I))=43$ and from Proposition \ref{sp}, we get $\reg G(I) \leq r_{J}(I)-1+(e_{2}(I)-1)e_{2}(I)-e_{3}(I)=55$. The bound given by Rossi, Trung and Valla is  $(e(I)-1)e(I)=6480$.
     
        \item  Let $R=\mathbb{Q}[[x,y]]$ and $I=(x^7,x^6y,x^3y^4,x^2y^5,y^7)$. By CoCoA, the Hilbert series is
        \begin{equation*}
              h_{I}(t)=\frac{32+11t+7t^{2}-t^{4}}{(1-t)^2}.
        \end{equation*}
           Then $ e(I)=49,$ $ e_{1}(I)=21,$ $e_{2}(I)=1$ and $e_{3}(I)=-4.$ By \cite[Theorem 2.1]{elias}, we have $\widetilde{I}=(x^7,x^6y,x^5y^2,x^4y^3,x^3y^4,x^2y^5,y^7)$ and $\lambda(\widetilde{I}/I)=3.$ Here $J=(x^7, x^6y+y^7)$ is a minimal reduction of $I$ with $r_{J}(I)=2.$ Note that $x^{11}y^{3} \in (I^3:I) \subseteq \widetilde{I^{2}}$ but $x^{11}y^{3} \notin I^2$, therefore,  $\widetilde{I^{2}} \neq I^{2}$. Thus  from Proposition \ref{lama}, we get $\reg G(I) \leq r_{J}(I)-1-e_{3}(I)+\left(\frac{r_{J}(I)-1}{2}\right)(e_{2}(I)-e_{1}(I)+e(I)-\lambda(R/I)+\lambda(\widetilde{I}/I))=5$ and from Proposition \ref{sp},  we get $\reg G(I) \leq r_{J}(I)-1-e_{3}(I)=5$. The bound given by Rossi, Trung and Valla is  $(e(I)-1)e(I)=2352$.
    \end{enumerate}
\end{example}

In  \cite[Theorem 4.4]{chl}, Linh proved that in a   Noetherian local  ring $R$ of dimension $d\geq 2,$
\begin{equation*}
    \reg G(I) \leq 2^{(d-1)!} D(I,R)^{3(d-1)!-1}-1,
\end{equation*}
where  $D(I,R)$ is the arbitrary \textit{extended degree} of $R$ with respect to $I.$ By Vasconcelos \cite{vas1}, it can be verified that, if $(R,\mathfrak m)$ is a $d$-dimensional Buchsbaum ring then 
\begin{equation*}
    D(I,R)=e(I)+I(R),
\end{equation*}
where,  the number $I(R)$ is called the \textit{Buchsbaum invariant} of $R$ and is defined in \cite[Chapter 2, Proposition 2.7]{sv}, as
\begin{equation*}
    I(R)=\displaystyle \sum_{i=0}^{d-1}\binom{d-1}{i}\lambda\left(\mathcal{H}_{\mathfrak m}^{i}(R)\right) \text{ or } I(R)=\displaystyle \sum_{i=1}^{d}(-1)^{i}e_{i}(J).
\end{equation*}
Therefore, for a two-dimensional Buchsbaum ring from Linh's result, we get
\begin{equation*}
    \reg G(I) \leq 2(e(I)+I(R))^{2}-1.
\end{equation*}

We conclude this section with the following  proposition, in which we give a bound on regularity in two-dimensional Buchsbaum rings. We also consider Example \ref{example 4.6} to compare  the bound in the following proposition with the one given by Linh in \cite[Theorem 4.4]{chl}.

\begin{proposition} \label{prop 5.8}

 Let $(R, \mathfrak m)$ be a two dimensional Buchsbaum ring with  $\depth R >0$  and $\mathbb S$ be the $S_{2}$-fication of $R$. Let $I$ be an $\mathfrak m$-primary ideal with minimal reduction $J$ and $I=I\mathbb S$. If $r_{J}(I) < \rho(I),$ then $\reg G(I) \leq  r_{J}(I)-1+(e_{2}(I)-e_{1}(J)-1)(e_{2}(I)-e_{1}(J))-e_{3}(I)$.
\end{proposition}

\begin{proof}
    From \cite[Theorem 2.4]{rtt}, we have $\reg G(I)=\max\{r_{J}(I), \rho(I)\}$. Since $r_{J}(I) < \rho(I),$ therefore, $\reg G(I)=\rho(I).$ Thus, from Theorem \ref{rudra}, we get $\reg G(I) \leq  r_{J}(I)-1+(e_{2}(I)-e_{1}(J)-1)(e_{2}(I)-e_{1}(J))-e_{3}(I)$.
\end{proof}

\begin{example}
    From Example \ref{example 4.6}, we have $r_{Q}(I) < \rho(I),$ therefore, from Proposition \ref{prop 5.8}, we get  $\reg G(I)\leq r_{J}(I)-1+(e_{2}(I)-e_{1}(J)-1)(e_{2}(I)-e_{1}(J))-e_{3}(I)=16. $ Now from \cite[Example 2.8]{go}, we get  $x-z, y-w$ is a system of parameters for $R.$ Let $J=(x-z,y-w),$ then $e_{1}(J)=-1$ and $e_{2}(J)=0$. Therefore, $I(R)=-e_{1}(J)+e_{2}(J)=1.$  Note that the bound $ 2(e(I)+I(R))^{2}-1=2(50+1)^{2}-1=5201$ given by Linh in \cite[Theorem 4.4]{chl}  is larger than our bound.
\end{example}

\section*{Acknowledgement}
We would like to express our sincere thanks to the referee for a meticulous
reading of the manuscript and making excellent suggestions, incorporating which made the article  read well.
We are also thankful to Prof. T. J. Puthenpurakal for his enlightening discussions and  valuable suggestions. Additionally, the second author acknowledges the support received from the Government of India through the Prime Minister's Research Fellowship during the course of this work.

 {}

\end{document}